\documentclass[11pt,reqno]{amsart}
\usepackage[utf8]{inputenc}
\usepackage{enumitem}
\usepackage{amssymb}
\usepackage{mathtools}
\usepackage{pdfsync}
\usepackage[english]{babel}
\usepackage[pdftex,pagebackref,colorlinks=true,urlcolor=blue,linkcolor=blue,citecolor=red]{hyperref}
\usepackage[capitalize]{cleveref}
\usepackage{fullpage}
\usepackage{color}
\usepackage{amsmath}
\usepackage{amsfonts}
\usepackage{mathrsfs}
\usepackage{t1enc, graphicx}
\usepackage{verbatim}
\usepackage{bbm}
\usepackage[colorinlistoftodos]{todonotes}
\usepackage{enumitem}
\usepackage{bm}
\usepackage{upgreek}
\usepackage[mathcal]{eucal}
\usepackage{tikz-cd}
\usepackage{adjustbox}
\usepackage{marginnote}
\usepackage[nocompress,noadjust]{cite}

\newcommand{\Z}{\mathbb{Z}}
\newcommand{\N}{\mathbb{N}}

\renewcommand{\P}{\bold{P}}

\newcommand{\RP}{\bold{RP}}

\newcommand{\bQ}{\bm{Q}}

\newcommand{\bx}{{\bold x}}
\newcommand{\by}{{\bold y}}
\newcommand{\bg}{{\bold g}}
\newcommand{\bh}{{\bold h}}
\newcommand{\bt}{{\bold t}}
\newcommand{\bs}{{\bold s}}
\newcommand{\bp}{{\bold p}}
\newcommand{\bq}{{\bold q}}
\newcommand{\bv}{{\bold v}}

\def\scr{\mathscr }

\def\GG{{\mathfrak G}}

\newcommand{\cltau}[1]{\mathrm{cl}_{\uptau}(#1)}
\newcommand{\HK}[1]{\mathcal{HK}^{[#1]}(G)}
\newcommand{\taulim}{\xrightarrow{\uptau}}
\newcommand{\Gtau}[1]{G_{#1}^{\uptau\mathrm{-top}}}

\newcommand{\ilim}[1]{\underset{\leftarrow}{\lim}#1 ~}

\newtheorem{theorem}{Theorem}[section]
\newtheorem{proposition}[theorem]{Proposition}

\newtheorem{lemma}[theorem]{Lemma}

\newtheorem{corollary}[theorem]{Corollary}
\newtheorem*{theorem*}{Theorem}
\theoremstyle{definition}
\newtheorem{definition}[theorem]{Definition}
\newtheorem*{definition*}{Definition}

\newtheorem{question}[theorem]{Question}
\theoremstyle{remark}

\newtheorem{remark}[theorem]{Remark}

\author{Axel \'Alvarez}
\address[Axel \'Alvarez]{Departamento de Ingenier\'{\i}a Matem\'atica, Universidad de Chile, Beauchef 851, Santiago, Chile} \email{aalvarez@dim.uchile.cl}

\author{Sebasti\'an Donoso}
\address[Sebasti{\'a}n Donoso]{Departamento de Ingenier\'{\i}a Matem\'atica and Centro de Modelamiento Matem{\'a}tico, Universidad de Chile \& IRL 2807 - CNRS, Beauchef 851, Santiago, Chile} \email{sdonosof@uchile.cl}

\thanks{The first author was supported by ANID-Subdirección de Capital
Humano/Magíster Nacional/2023-22232091 and ANID-Subdirección de Capital Humano/Doctorado Nacional/2025-21251865. The second author was partially funded by ANID/Fondecyt/1241346. Both authors were partially funded by Centro de Modelamiento Matemático (CMM) FB210005, BASAL funds for centers of excellence from ANID-Chile and ECOS-ANID grant C21E04 (ECOS210033).}



\title{Cube structures of the universal minimal system, nilsystems and applications}

\date{}

\begin{document}



\maketitle
\begin{abstract} 
We propose and develop an approach to study nilsystems and their proximal extensions using cube structures associated with the universal minimal system. We provide alternative proofs for results regarding saturation properties of factor maps to maximal nilfactors in cubes, as well as new results and applications of independent interest to the structural theory of topological systems. In particular, we give a new proof that $\RP^{[d]}$ is an equivalence relation, building upon the distal case, by establishing a description of this relation in algebraic terms. This is new even for $d=1$.   
\end{abstract}

\section{Introduction}

A topological dynamical system is a pair $(X, T)$, where $X$ is a compact metric space and $T$ is a group acting on $X$. In this paper, we mainly focus on abelian group actions.  Minimal topological systems possess a rich and vast theory. This theory provides a classification based on fundamental system types and extensions, allowing one to construct any given minimal system using towers of extensions and factors. For a comprehensive exposition of this theory, we refer the reader to \cite{Auslander_minimal_flows_and_extensions:1988,deVries_elements_topological_dynamics:1993}.
Key concepts of this theory are the equicontinuous systems and equicontinuous/proximal extensions. Minimal abelian equicontinuous systems, that is, where the group forms an abelian equicontinuous family of transformations, can be characterized through the triviality of a closed, invariant, equivalence relation, called the equicontinuous regionally proximal relation, and usually denoted by $\RP$. That is, minimal equicontinuous systems can be characterized as $\RP=\Delta$, where $\Delta$ is the diagonal $\{(x,x):x\in X\}$. Moreover, any minimal topological dynamical system admits a maximal equicontinuous factor, which is, in fact, its quotient by the $\RP$ relation (see \cite[Chapter 9]{Auslander_minimal_flows_and_extensions:1988} for a more in-depth discussion). 

In recent years, there has been much interest in the study of nilsystems and their inverse limits, since they relate to many dynamical properties such as being characteristic and have important applications in number theory and additive combinatorics. (see for example \cite[Chapter 1]{Host_Kra_nilpotent_structures_ergodic_theory:2018} and the references therein). This class, a subclass of distal systems, strictly includes equicontinuous systems while keeping many of the important and useful properties of equicontinuous systems (see for instance \cite[Chapter 11]{Host_Kra_nilpotent_structures_ergodic_theory:2018}). 

In their pioneering work, Host, Kra and Maass in \cite{Host_Kra_Maass_nilstructure:2010} introduced the higher order regionally proximal relations $\RP^{[d]}$ for $\Z$-actions, extending available theory for equicontinuous systems, which allows us to precisely characterize nilsystems and their inverse limits (see \cref{sec:nil_defs} for the precise definitions). They showed, among other things, that in a distal minimal system, $\RP^{[d]}$ is an equivalence relation and that the quotient by it is the maximal system of order $d$ (that is, an inverse limit of $d$-step nilsystems). 

Soon after, Shao and Ye in \cite{Shao_Ye_regionally_prox_orderd:2012} extended these results by proving that $\RP^{[d]}$ is an equivalence relation in any minimal $\Z$-system (and even for a general abelian group action), and that the quotient by this relation is the maximal system of order $d$. \footnote{For general group actions, regional proximality of higher order can be defined \cite{Glasner_Gutman_Ye_higher_regionallyproximal_general_groups:2018}.}  

It is classical that for abelian group actions, equicontinuity can be characterized through algebraic properties of its enveloping semigroup: indeed, they are characterized by having an abelian enveloping semigroup. It is natural to ask if there is any algebraic property, or even characterization, of nilsystems and their inverse limits. This is indeed the case:  it was proven in \cite{Donoso_enveloping_systems_orderd:2014, Qiu_Zhao_topnilpotent_enveloping_nil:2022} that the property of having an enveloping semigroup that is a $d$-step topological semigroup characterizes minimal systems of order $d$ for $\Z$-actions.

All the aforementioned results belong to the class of distal systems, where it is known that the enveloping semigroup is a group. In the structural theory of minimal systems, proximal extensions play a central role and from many aspects, a proximal extension of a system behaves similarly to the factor system. However, a proximal extension of a distal system is never distal (unless the extension is an isomorphism). So, even though the system is dynamically similar to the factor, all the nice algebraic structure of the enveloping semigroup might be lost under a proximal extension. In particular, the enveloping semigroup is no longer a group.

A key aspect of the structural theory of minimal systems is the development of the so-called algebraic theory, or Galois theory, of minimal systems, proposed by Ellis and developed by several authors during several decades \cite{Ellis_Glasner_Shapiro_PI-flows:1975, Ellis_lectures_topological_dynamics:1969, Ellis_the_Veech_struct_thm:1974, Auslander_Ellis_group_quasifactor:2000, Glasner_quasifactor_minimal_systems:2000, Ellis_Glasner_Pure_weak_mixing:1978, Auslander_Ellis_Ellis_Regionally_proximal:1995, Akin_Glasner_topological_descomposition_homogenous:1998, Glasner_top_erg_decomposition:1994}. This theory allows us to use tools of an algebraic nature to study the topological systems which are not necessarily distal. In particular, this theory provides algebraic characterizations of several types of extensions (proximal, RIC, etc.). (See, for instance, \cite[Chapters V and VI]{deVries_elements_topological_dynamics:1993} for an exhaustive exposition).

In this paper, we are interested in understanding proximal extensions of nilsystems from an algebraic point of view. To do so, we study cubic structures, Host-Kra groups and the Ellis group of a pointed minimal system and study which properties of them are related to nilsystems and their inverse limits. 

Our first result is that we can characterize proximal extensions of nilsystems in terms of unique completion of a cubic structure. It is worth noting that the unique completion property of a cube structure appears as a fundamental property in many previous works (for instance in \cite{Host_Kra_nonconventional_averages_nilmanifolds:2005,Host_Kra_Maass_nilstructure:2010,Donoso_Sun_cubes_product_ext:2015,Cabezas_Donoso_Maass_directional_cubes:2020, Gutman_Manners_Varju_nilspaces_III:2020,Glasner_Gutman_Ye_higher_regionallyproximal_general_groups:2018}). For the statement of the next theorem, we refer to \cref{sec:algebric_cubes} for the precise definitions. 
\begin{theorem*}[\cref{thm: unique_tauclosure_iff_prox_ext_nil}]
    Let $(X, T)$ be a minimal topological dynamical system,  where $T$ is an abelian group and $d\geq 1$ an integer. There is a group, denoted \footnote{ The group $G$ is properly defined in \cref{sec:univ_system} through the universal minimal system. The abbreviation $\mathcal{HK}$ stands for {\em Host-Kra} group.} $\HK{d}$, that has the property of unique $\uptau$-closure on $X$ if and only if $(X, T)$ is a proximal extension of a system of order $d$.
\end{theorem*}

As a byproduct, we give an alternative definition of the regionally proximal relations in terms of algebraic objects, that extends similar definitions in the case of systems of order $d$ introduced in \cite{Donoso_enveloping_systems_orderd:2014} and developed in \cite{Qiu_Zhao_topnilpotent_enveloping_nil:2022}. The precise definitions of the algebraic objects mentioned below are given later. More precisely, see \cref{sec:univ_system} for the definition of $M$, $J(M)$,  \cref{def:tau_commutator} for that of $ \Gtau{d+1}$ and \cref{sec:applic_RP} for the definition of $D$.

\begin{theorem*}[\cref{lemma: Rdis_is_RPd}] \label{lemma:intro_Rdis_is_RPd}    
    Let $(X,\Z)$ be a minimal topological dynamical system and let $d\geq 1$ be an integer. Then $\RP^{[d]}(X)=\{(x,vghx): v\in J(M), g\in \Gtau{d+1},h\in D, x\in X\}.$
\end{theorem*}

\subsection{Applications} 
\subsubsection*{$\RP^{[d]}$ is an equivalence relation.}
Regarding the $\RP^{[d]}$ relation, in \cref{sec:Regionally_proximal}, we study it through the lens of the algebraic objects we described, and in \cref{sec:applications_charfac_cube}, we derive a new proof that it is an equivalence relation for an action of $\Z$ (\cref{lemma: Rdis_is_RPd}). Actually, this can be extended to any group action in which having a (topologically) $d$-step nilpotent enveloping semigroup is equivalent to being a system of order $d$. For the group $\Z$, this characterization was established in \cite{Qiu_Zhao_topnilpotent_enveloping_nil:2022}, but intriguingly enough, it is not known if such characterization holds for $\Z^d$, $d>1$, let alone general abelian group actions. Nonetheless, in \cite{Qiu_Zhao_topnilpotent_enveloping_nil:2022}, the authors succeeded in showing that having a (topologically) nilpotent enveloping semigroup implies being a system of order $\infty$.

Using the algebraic description, we also provide a simple and direct proof of the {\em lifting property} of the $\RP^{[d]}$ relation (\cref{cor: lifting_RPd}).

\subsubsection*{Characteristic factors along cubes}
In the study of multiple ergodic averages, characteristic factors play a crucial role. The notion of the characteristic factor was first introduced in a paper by Furstenberg and Weiss \cite{Furstenberg_Weiss_ergodic_thm_double:1996}, and its relevance was solidified in the groundbreaking work of Host and Kra \cite{Host_Kra_nonconventional_averages_nilmanifolds:2005}. A counterpart of the notion of characteristic factors in a topological dynamical system was first studied by Glasner \cite{Glasner_top_erg_decomposition:1994}. Later, Cai and Shao defined the notion of characteristic factors along cubes \cite{Cai_Shao_Topological_characteristic_cubes:2019}.

In \cref{sec:applications_charfac_cube}, we extend \cite[Theorem 1.7]{Cai_Shao_Topological_characteristic_cubes:2019} using our machinery. In particular, we prove that, modulo almost one-to-one extensions, the maximal factor of order $d$ is a characteristic topological factor along cubes of order $d-1$. Furthermore, we extend these results for the maximal factor of order $\infty$. Using this saturation theorem, we prove that the maximal distal factor of a dynamical cube of a system is the dynamical cube of the maximal distal factor of the system.

It is worth noting that the notion of characteristic factors in topological dynamics has been extensively explored in recent years (see \cite{Glasner_Huang_Shao_Weiss_Ye_Topological_characteristic_factors:2020}), and it has given tools to address several problems in topological dynamics (see for instance \cite{Cai_Shao_top_charact_independence:2022,Glasscock_sim_approx_nil_thick:2024, Huang_Shao_Ye_top_induced_poly_comb:2023, Donoso_Koutsogiannis_Sun_joint_trans_linear:2024}).

\subsubsection*{Structure theorems for cubes}
In 1963, Furstenberg in \cite{Furstenberg_structure_distal_flows:1963} established the structure theorem for minimal distal systems. Veech in \cite{Veech_point-distal_flows:1970} established the structure theorem for pointed distal minimal systems. Finally, Ellis, Glasner, and Shapiro in \cite{Ellis_Glasner_Shapiro_PI-flows:1975}, McMahon in \cite{McMahon_Weak_mixing_structure_thm:1976}, and Veech in \cite{Veech_topological_dynamics:1977} gave the structure theorem for general minimal systems. Regarding structures associated with a minimal system, in \cite{Wu_Xu_Ye_Structure_saturated:2023}, the authors showed that, for any $(X,\Z)$ minimal system, the system $N_{d}(X)$ has the same structure theorem as $X$ (see \cite[Section 2]{Wu_Xu_Ye_Structure_saturated:2023} for the definition of the system $N_{d}(X)$). In this paper, we will present a similar result for dynamical cubes using the algebraic properties developed.

\begin{theorem*}[\cref{thm: Structure_thm_cube}]
    Let $(X, T)$ be a minimal topological dynamical system and let $d\geq 1$ be an integer, where $T$ is an abelian group. Then, $(\bQ^{[d]}(X),\mathcal{HK}^{[d]}(T))$ has the same structure theorem as $(X, T)$. In particular, if $(X, T)$ is PI (resp. I, AI), then so is $(\bQ^{[d]}(X),\mathcal{HK}^{[d]}(T))$.
\end{theorem*}

\section{Background} \label{Sec:background}

\subsection{Topological dynamical systems} A topological dynamical system (or just a system) is a pair $(X, T)$, where $X$ is a compact Hausdorff space and $T$ is a discrete group acting as a group of homeomorphisms of the space $X$. Unless otherwise stated, we will assume in this paper that $T$ is abelian and $X$ is a compact metric space. For convenience, we write $(X,S)$ instead of $(X,T)$ whenever the group $T$ is generated by a single homeomorphism.  If $(X_{i},T)_{i\in I}$ is a family of systems, the action of $T$ on the product space $\prod_{i\in I} X_{i}$ is defined coordinatewise: $t(x_{i})_{i\in I} = (tx_{i})_{i\in I}$ for each $t\in T$ and $(x_{i})_{i\in I} \in\prod_{i\in I} X_{i}$.

The system is {\em transitive} if there is a point $x\in X$ such that its orbit $Tx:=\{tx: t\in T\}$ is dense in $X$. The system is {\em minimal} if the orbit of any point is dense in $X$. A pair of points $(x,y)$ in $X\times X$ is {\em proximal} if there is a net $(t_{\lambda})_{\lambda\in\Lambda}$ in $T$ and a point $z\in X$ such that $\lim t_{\lambda}x = \lim t_{\lambda}y=z$ and $(x,y)$ is a {\em distal} pair if it is not proximal. A point $x\in X$ is a {\em distal point} if it is proximal only to itself. The set of proximal pairs is denoted by $\P(X)$, and called the {\em proximal relation}.

An onto continuous map $\pi\colon (X, T) \to (Y, T)$ between the topological dynamical systems is an {\em extension} (or a {\em factor}) if $t\pi(x)=\pi(tx)$ for every $x\in X$ and $t\in T$. We say that the extension is proximal (distal) if, for every  $y \in Y$, every pair of points in $\pi^{-1}(y)$ is proximal (distal). We say that $\pi$ is {\em almost one-to-one} if there exists a dense $G_{\delta}$ set $X_{0} \subseteq X$ such that $\pi^{-1}(\pi(x))=\{x\}$ for any $x\in X_{0}$. If an extension $\pi\colon (X, T) \to (Y, T)$ between minimal systems is an almost one-to-one extension, then it is also a proximal extension ({\cite[Lemma 2.1]{Donoso_Durand_Maass_Petite_automorphism_low_complexity:2016}}).

An important class of topological dynamical systems is the class of equicontinuous systems. A system $(X, T)$ is {\em equicontinuous} if the collection of maps defined by the group $T$ is an equicontinuous family. These systems can be characterized by the {\em regional proximal relation}. A pair $(x,y)\in X\times X$ is said to be regionally proximal if there are nets $(x_{\lambda})_{\lambda\in\Lambda},(y_{\lambda})_{\lambda\in\Lambda}$ in $X$ and $(t_{\lambda})_{\lambda\in\Lambda}$ in $T$ with $x_{\lambda}\to x$, $y_{\lambda}\to y$ and $(t_{\lambda}x_{\lambda},t_{\lambda}y_{\lambda})\to (z,z)$, for some $z\in X$. The set of regionally proximal pairs is denoted by $\RP(X)$ and is the {\em regionally proximal relation}.

One can relativize the notion of equicontinuity. We say that an extension $\pi\colon (X, T)\to (Y, T)$ is equicontinuous if for any $\varepsilon>0$, there exists $\delta>0$ such that if $(x,y)\in R_{\pi}$ and $d(x,y)<\delta$ then $d(tx,ty)<\varepsilon$ for all $t\in T$. Here, $R_{\pi} = \{(x,y)\in X\times X:\pi(x)=\pi(y)\}$.

\subsection{Nilpotent groups, nilmanifolds and nilsystems} \label{sec:nil_defs} Let $L$ be a group. For $g,h\in L$ we write $[g,h]=g^{-1}h^{-1}gh$ for the commutator of $g$ and $h$ and for $A,B\subseteq L$ we write $[A,B]$ for the subgroup spanned by $\{[a,b]:a\in A,b\in B\}$. The commutator subgroups $L_{j}$, $j\geq 1$, are defined inductively by setting $L_{1}=L$ and $L_{j+1}=[L_{j},L]$. Let $d\geq 1$ be an integer. We say that $L$ is {\em $d$-step nilpotent} if $L_{d+1}$ is the trivial subgroup.

Since we will work with groups that are also topological spaces, we can also consider a topological definition of nilpotency, which will be more suitable for our purposes. Let $L$ be a topological space that has a group structure. (Here, we do not assume the group to be a topological group; in fact, in many cases of interest, it will not be.)  For $A,B \subseteq L$, we define $[A,B]_{\mathrm{top}}$ as the closed subgroup spanned by $\{[a,b] : a \in A, b \in B\}$. Note that $[A,B]\subseteq [A,B]_{\mathrm{top}}$ for all $A,B\subseteq L$. The topological commutator subgroups $L^{\mathrm{top}}_j$, $j \geq 1$, are similarly defined as $L^{\mathrm{top}}_1 = L$ and $L^{\mathrm{top}}_{j+1} = [L^{\mathrm{top}}_j, L]_{\mathrm{top}}$. For $d \geq 1$ an integer, we say that $L$ is {\em $d$-step topologically nilpotent} if $L^{\mathrm{top}}_{d+1}$ is the trivial subgroup. Since $L_j \subseteq L^{\mathrm{top}}_j$ for every $j \geq 1$, we have that if $L$ is $d$-step topologically nilpotent, then $L$ is also $d$-step nilpotent.

Let $L$ be a $d$-step nilpotent Lie group and $\Gamma$ a discrete cocompact subgroup of $L$. The compact manifold $X = L/\Gamma$ is called a {\em $d$-step nilmanifold}. We recall here that since $L$ is a nilpotent Lie group, the commutator subgroups are closed, and then, in this case, the notions of $d$-step nilpotent and $d$-step topologically nilpotent coincide.

Observe that $L$ acts naturally on $X$ by left translation. If $T$ is a topological group and $\phi: T\to L$ is a continuous homomorphism, the induced action $(X, T)$ is called a {\em nilsystem of order $d$}.

\subsection{Host-Kra cube groups associated with a group}

Let $d \geq 1$ be an integer, and write $[d] = \{1, 2, \dots, d\}$. We view an element of $\{0,1\}^{d}$, the Euclidean cube, either as a sequence $\epsilon = (\epsilon_{1}, \dots, \epsilon_{d})$ of 0's and 1's; or as a subset of $[d]$. A subset $\epsilon$ corresponds to the sequence $(\epsilon_{1}, \dots, \epsilon_{d}) \in \{0,1\}^{d}$ such that $i \in \epsilon$ if and only if $\epsilon_{i} = 1$ for $i \in [d]$. For example, $\overrightarrow{0} = (0, \dots, 0) \in \{0,1\}^{d}$ is the same as $\emptyset \subset [d]$ and $\overrightarrow{1} = (1, \dots, 1)$ is the same as $[d]$.

If $X$ is a set, we denote $X^{2^{d}}$ by $X^{[d]}$ and we write a point $\bx \in X^{[d]}$ as $\bx = (x_{\epsilon} : \epsilon \subseteq [d])$. For example, for $d=2$ we have $\bx=(x_{\emptyset},x_{\{1\}},x_{\{2\}},x_{\{1,2\}})$. Sometimes, it is convenient to view an element of $X^{[d]}$ as a map from $\{0,1\}^{d}$ to $X$. For a point $x\in X$ we let $x^{[d]}\in X^{[d]}$ be the diagonal point all of whose coordinates are $x$.

Let $0\leq \ell\leq d$ be an integer. An {\em $\ell$-dimensional face} of $\{0,1\}^{d}$, or equivalently a {\em face of codimension $d-\ell$} of $\{0,1\}^{d}$, is a subset of $\{0,1\}^{d}$ obtained by fixing the values of $d-\ell$ coordinates. We write dim($\alpha$) and codim($\alpha$) for the dimension and codimension of a face $\alpha$. A {\em facet} of $\{0,1\}^{d}$ is defined to be a face of codimension $1$. A face $\alpha$ is an {\em upper face} if $\overrightarrow{1}\in\alpha$.

Let $L$ be a group and $d\geq 1$ be an integer. If $\alpha$ is a face of $\{0,1\}^{d}$, for $g\in L$ we define $g^{(\alpha)}\in L^{[d]}$ by\begin{align*}
    g^{(\alpha)}(\epsilon) = \left\lbrace \begin{matrix}
        g & \text{if } \epsilon\in\alpha\\
        e & \text{otherwise}.
    \end{matrix}\right.
\end{align*}

where $e$ is the identity element of $L$.

We call the subgroup of $L^{[d]}$ generated by all $g^{(\alpha)}$, where $g\in L$ and $\alpha$ is a facet of $\{0,1\}^{d}$, the {\em Host-Kra cube group} (of order $d$) associated with $L$ and denote it by $\mathcal{HK}^{[d]}(L)$. We call the subgroup of $L^{[d]}$ generated by all $g^{(\alpha)}$, where $g\in L$ and $\alpha$ is an upper facet of $\{0,1\}^{d}$, the {\em face cube group} and denote it by $\mathcal{F}^{[d]}(L)$.

The Host-Kra and face cube groups originally appeared in \cite[Section 5]{Host_Kra_nonconventional_averages_nilmanifolds:2005} and coincide with the parallelepiped groups and face groups respectively of {\cite[Definition 3.1]{Host_Kra_Maass_nilstructure:2010}} introduced for abelian actions. See also \cite[Appendix E]{Green_Tao_linear_eq_primes:2010} for treatment of Host-Kra cube groups in nilpotent Lie groups.

The following proposition illustrates the connection between these two groups.

\begin{proposition}[{\cite[Proposition 3.3]{Glasner_Gutman_Ye_higher_regionallyproximal_general_groups:2018}}] 
    Let $L$ be a group and $d\geq 1$ be an integer. Then, $\mathcal{HK}^{[d]}(L)=\mathcal{F}^{[d]}(L)\Delta^{[d]}(L)$, where $\Delta^{[d]}(L) = \{g^{[d]}:g\in L\}$.
\end{proposition}

The following proposition gives a more precise description of the groups $\mathcal{HK}^{[d]}(L)$ by introducing a particular parametrization.

\begin{proposition}[see {\cite[Chapter 6]{Host_Kra_nilpotent_structures_ergodic_theory:2018}}]\label{prop: parametrization_cubes}
    Let $d\geq 1$ and let $L$ be a group. Suppose that $\alpha_{1},\dots,\alpha_{2^{d}}$ is an enumeration of all upper faces of $\{0,1\}^{d}$ such that codim($\alpha_{j}$) does not decrease with $j$. Then the map\begin{align*}
        \Psi:\prod_{j=1}^{2^{d}}L_{\mathrm{codim}(\alpha_{j})}\to L^{[d]}
    \end{align*}

    given by\begin{align*}
        \Psi(\bold{h})=\prod_{j=1}^{2^{d}}\bold{h}_{j}^{(\alpha_{j})},
    \end{align*}

    where $L_{0}=L$ and the product is taken in increasing order of indices, is a bijection of $\prod_{j=1}^{2^{d}}L_{\mathrm{codim}(\alpha_{j})}$ onto $\mathcal{HK}^{[d]}(L)$.
\end{proposition}

\subsection{Dynamical cubes} Let $(X,T)$ be a topological dynamical system. The Host-Kra cube of $T$ acts on $X^{[d]}$ by\begin{align*}
    (\bold{t}\bx)(\epsilon)=\bt_{\epsilon}\bx_{\epsilon}
\end{align*}

for $\bt\in \mathcal{HK}^{[d]}(T), \bx\in X^{[d]}$ and $\epsilon\subseteq [d]$.

Let $d\geq 1$ be an integer. Following Host, Kra and Maass {\cite[Definition 1.1]{Host_Kra_Maass_nilstructure:2010}} the set of {\em dynamical cubes of dimension $d$} is \begin{align*}
    \bQ^{[d]}(X)=\overline{\{\bt x^{[d]}:\bt\in\mathcal{HK}^{[d]}(T),x\in X\}}.
\end{align*}

Note that if $(X,T)$ is minimal \begin{align*}
        \bQ^{[d]}(X)= \overline{\{\bt x_{0}^{[d]}:\bt\in\mathcal{HK}^{[d]}(T)\}} = \overline{\{\bt x^{[d]}:\bt\in\mathcal{F}^{[d]}(T),x\in X\}}.
    \end{align*}

for each $x_{0}\in X$.

\begin{theorem}[{\cite[Theorem 1.1]{Glasner_RPd_envelop_proof:2014}, \cite[Lemma 3.2]{Host_Kra_Maass_nilstructure:2010}, \cite[Theorem 3.1]{Shao_Ye_regionally_prox_orderd:2012}}\label{thm: cube_is_minimal}]
    Let $(X, T)$ be a minimal topological dynamical system and $d\geq 1$ be an integer. Let\begin{align*}
        \bQ_{x}^{[d]}(X) &= \bQ^{[d]}(X) \cap (\{x\}\cap X^{2^{d}-1})\\
        Y_{x}^{[d]} &= \overline{\mathcal{F}^{[d]}(T)x^{[d]}},
    \end{align*}

    where $x\in X$. Then,\begin{enumerate}[label=(\arabic*)]
        \item $(\bQ^{[d]}(X),\mathcal{HK}^{[d]}(T))$ is a minimal topological dynamical system.
        \item For each $x\in X$, $(Y_{x}^{[d]},\mathcal{F}^{[d]}(T))$ is a minimal topological dynamical system.
        \item For each $x\in X$, $Y_{x}^{[d]}$ is the unique minimal subsystem of $(\bQ_{x}^{[d]}(X),\mathcal{F}^{[d]}(T))$.
        \item There exists a dense $G_{\delta}$ subset $X_{0}\subseteq X$ such that $Y_{x}^{[d]} = \bQ_{x}^{[d]}(X)$ for every $x\in X_{0}$.
    \end{enumerate}
\end{theorem}

 The proof of \cref{thm: cube_is_minimal} depends on previously defined algebraic objects, especially the enveloping semigroup, and similar arguments can be made using the universal minimal system.

\subsection{Regionally proximal relation of order $d$}

Let $(X, T)$ be a topological dynamical system and $d\geq 1$ be an integer. A pair $(x,y)\in X\times X$ is said to be {\em regionally proximal of order $d$} if  there are nets $(\bold{f}_{\lambda})_{\lambda\in\Lambda}$ in $\mathcal{F}^{[d]}(T)$, $(x_{\lambda})_{\lambda\in\Lambda}, (y_{\lambda})_{\lambda\in\Lambda}$ in $X$ and $\bold{a}_{*}\in X^{2^{d}-1}$ such that\begin{align*} (\bold{f}_{\lambda}x_{\lambda}^{[d]},\bold{f}_{\lambda}y_{\lambda}^{[d]}) \to (x,\bold{a}_{*},y,\bold{a}_{*})
\end{align*}
The set of regionally proximal pairs of order $d$ is denoted by $\RP^{[d]}(X)$ and is called the {\em regionally proximal relation of order $d$}. We say that $(X, T)$ is a {\em system of order $d$} if $\RP^{[d]}(X)$ coincides with the diagonal relation. Note that $\RP^{[1]}(X)$ is the classical regionally proximal relation, so systems of order 1 are exactly the equicontinuous systems. 

The relation $\RP^{[d]}(X)$ is a closed and invariant relation. Moreover, (see \cite[Lemma 3.7]{Shao_Ye_regionally_prox_orderd:2012})\begin{align*}
    \P(X)\subseteq \cdots \subseteq \RP^{[d+1]}(X)\subseteq \RP^{[d]}(X)\subseteq \cdots \subseteq \RP^{[1]}(X). 
\end{align*}

The following theorem shows some properties of the regionally proximal relation of order $d$.\begin{theorem}[{\cite[Theorems 3.2, 3.3, 3.8 and 3.9]{ Shao_Ye_regionally_prox_orderd:2012}}]\label{thm: shao_ye_thm}
    Let $\pi:(X, T)\to (Y, T)$ be an extension of minimal systems and $d\geq 1$ be an integer. Then,\begin{enumerate}[label=(\arabic*)]
        \item $(x,y)\in\RP^{[d]}(X)$ if and only if $(x,y,\dots,y)\in \bQ^{[d+1]}(X)$ if and only if $(x,y,\dots,y)\in \overline{\mathcal{F}^{[d+1]}(T)x^{[d+1]}}$.
        \item $\RP^{[d]}(X)$ is an equivalence relation.
        \item $\pi\times\pi(\RP^{[d]}(X))=\RP^{[d]}(Y)$.
        \item $\RP^{[d]}(X)\subseteq R_{\pi}$ if and only if $(Y, T)$ is a system of order $d$.
    \end{enumerate}

    Furthermore the quotient of $X$ under $\RP^{[d]}(X)$ is the maximal system of order $d$. It follows that every factor of a system of order $d$ is a system of order $d$.
\end{theorem}

\subsection{Enveloping semigroups}

The {\em enveloping semigroup} (or {\em Ellis semigroup}) $E(X, T)$ of a system $(X, T)$ is defined as the closure of $T$ in $X^X$ endowed with the product topology. For an enveloping semigroup $E(X, T)$, the applications $E(X, T) \to E(X, T)$, $p \mapsto pq$ and $p \mapsto tp$ are continuous for all $q \in E(X,T)$ and $t\in T$.

Ellis introduced this notion, and it has proved to be a useful tool in studying dynamical systems. The algebraic properties of $E(X,T)$ can be converted into the dynamical properties of $(X, T)$ and the reverse is also true. To exemplify this, we recall the following theorem.

\begin{theorem}[see {\cite[Chapters 3, 4 and 5]{Auslander_minimal_flows_and_extensions:1988}}]\label{thm: distal_equi_chr_enveloping_semigroup}
    Let $(X, T)$ be a topological dynamical system. Then, \begin{enumerate}[label=(\arabic*)]
        \item $(X, T)$ is distal if and only if $E(X, T)$ is a group.
        \item $(X, T)$ is equicontinuous if and only if $E(X, T)$ is an abelian group if and only if $E(X, T)$ is a group of continuous transformations.
    \end{enumerate}
\end{theorem}

In particular, a system of order $1$ can be characterized by its enveloping semigroup. It is a natural question to ask whether it is possible to characterize the systems of order $d$ by their enveloping semigroup. The following theorem answers this question in the case of $\Z$-actions.

\begin{theorem}[{\cite[Theorem 1.1]{Donoso_enveloping_systems_orderd:2014},\cite[Theorem 1.3]{Qiu_Zhao_topnilpotent_enveloping_nil:2022}}]\label{thm: nil_iff_enveloping_nil}
    Let $(X, S)$ be a minimal topological dynamical system. Then, $(X, S)$ is a system of order $d$ if and only if $E(X, S)$ is a $d$-step topologically nilpotent group.
\end{theorem}
(Here we recall that we use the notation $(X,S)$ to refer to a $\Z$-action).

For a distal system, we let $(E_{d}^{\mathrm{top}}(X, T))_{d\in\N}$ denote the sequence of topological commutators of $E(X, T)$.

\begin{theorem}[{\cite[Theorem 1.4]{Qiu_Zhao_topnilpotent_enveloping_nil:2022}}]\label{thm: charc_RPd_Rd}
    Let $(X, S)$ be a minimal system of order $d$. Then, for $k=1,\dots ,d$, we have\begin{align*}
        \RP^{[k]}(X)=\{(x,px):x\in X, p\in E_{k+1}^{\mathrm{top}}(X, T)\}.
    \end{align*}
\end{theorem}

\subsection{The universal system} \label{sec:univ_system}
We refer to \cite[Sections IV, V and VI]{deVries_elements_topological_dynamics:1993} for the material discussed in this section. It is known that the semigroup $\beta T$, the Stone–Čech compactification of the discrete group $T$, is the universal point transitive system.  That is, for every transitive system $(X, T)$ and a point $x\in X$ with dense orbit, there exists an extension of systems $(\beta T, T)\to (X, T)$ which sends $e$, identity element of $T$, onto $x$. Therefore, by universality, there exists a unique extension $\Phi_{X}:(\beta T, T)\to (E(X, T), T)$, which is also a semigroup homomorphism, and we can interpret the $\beta T$ action on $X$ via this homomorphism. We say that two elements $p,q\in \beta T$ {\em coincide} on $X$ if $px=qx$ for all $x\in X$.

The semigroup $\beta T$ admits many minimal left ideals, which coincide with the minimal subsystems. All these ideals are isomorphic to each other both as compact right topological semigroups and as minimal systems. We will fix a minimal left ideal $M$ of $\beta T$.  The universality of $\beta T$ implies that $(M, T)$ is the universal minimal system. For a semigroup, the element $v$ with $v^{2}=v$ is called an {\em idempotent}. By the Ellis Namakura theorem, the set $J(M)$ of idempotents in $M$ is nonempty. Moreover, $vM$ is a group with identity $v$, where $v\in J(M)$. We define an equivalence relation in $J(M)$ by $u\sim v$ if $uv=v$ and $vu=u$. We fix one such idempotent $u\in J(M)$.

If $(X,T)$ is minimal, then $X=Mx$ for every $x\in X$. A necessary and sufficient condition for $x$ to be minimal is that $ux=x$ for some $u\in J(M)$. A minimal system $(X, T)$ is distal if and only if $X=vX$ for every $v\in J(M)$ (see \cite[Chapter 6]{Auslander_minimal_flows_and_extensions:1988}). So, a minimal system $(X,T)$ is distal iff $\Phi_{X}(M)=\Phi_{X}(vM)=E(X,T)$ for every $v\in J(M)$. Also, we can characterize the proximal points: a pair $(x,y)\in X\times X$ is proximal if and only if there exists a minimal left ideal $I$ of $\beta T$ such that $y=vx$ for some $v\in J(I)$. 

Let $2^{X}$ be the collection of nonempty closed subsets of $X$ endowed with the Hausdorff topology. The action of $T$ on $2^{X}$ is given by\begin{align*}
    tA=\{ta:a\in A\}
\end{align*}

for each $t\in T$ and $A\in 2^{X}$. This action induces another action of $\beta T$ on $2^{X}$, and we denote this action by the circle operation: $p\circ A$, where $p\in \beta T$ and $A\subset X$, to distinguish it from the subset $pA$. For $p\in \beta T$ and $A\subseteq X$, we define $p\circ A=p\circ \overline{A}$ and $p\circ \emptyset=\emptyset$, where $\overline{A}$ denotes the closure of $A$ in the usual topology of $X$. It holds that \begin{align*}
    p\circ A = \{x\in X: \exists (t_{\lambda})_{\lambda\in\Lambda}\subseteq T, \exists (x_{\lambda})_{\lambda\in\Lambda}\subseteq A\text{ with } t_{\lambda}\to p, t_{\lambda}x_{\lambda}\to x\}.
\end{align*}

for all $p\in \beta T$ and $A\in 2^{X}$,  where $t_{\lambda}\to p$ denotes the convergence of $(t_{\lambda})_{\lambda\in \Lambda}$ to $p$ in the usual topology of $\beta T$. We use this notation for convergence in the usual topology throughout the paper.

The set $G=uM$ is a subgroup of $M$ with identity $u$. The set $\{vM:v\in J(M)\}$ is a partition of $M$ and each $p\in M$ has a unique representation $p=vg$, where $v\in J(M)$ and $g\in G$. We sometimes write $p^{-1}$ for $vg^{-1}$. The group $G$ plays a central role in the algebraic theory of minimal systems. It carries a $T_{1}$ compact topology, called by Ellis the {\em $\uptau$-topology}, which is weaker than the relative topology induced on $G$ as a subset of $M$. This topology was first introduced by Furstenberg in \cite{Furstenberg_structure_distal_flows:1963}, and developed by Ellis, Glasner and Shapiro in \cite{Ellis_Glasner_Shapiro_PI-flows:1975}. For any subset $A\subseteq G$, the $\uptau$-topology is determined by\begin{align*}
    \cltau{A} = u(u\circ A) = (u\circ A) \cap G.
\end{align*}

Let $(g_{\lambda})_{\lambda\in\Lambda}$ be a net in $G$ and $g\in G$. The convergence of $(g_{\lambda})_{\lambda\in\Lambda}$ to $g$ with respect to the $\uptau$-topology is denoted by $g_{\lambda}\taulim g$. The $\uptau$-topology has many useful properties, some of which we state below for later purposes. 

\begin{proposition}[see {\cite[Chapter V]{deVries_elements_topological_dynamics:1993}}]\label{lemma: conv_prod_tau_top}
    Let $(X, T)$ be a minimal topological dynamical system, $(g_{\lambda})_{\lambda\in \Lambda}$ be a net in $G$ and $p,r\in M$.\begin{enumerate}[label=(\arabic*)]
        \item Both left and right multiplication by a fixed element of $G$ are homeomorphisms with respect to the $\uptau$-topology.
        \item Consider $(g_{\lambda})_{\lambda\in\Lambda}$ a net in $G$ such that $g_{\lambda} \to p$. Then $(g_{\lambda})_{\lambda\in \Lambda}$ has a subnet $(g_{\lambda_{i}})_{i\in I}$ such that $g_{\lambda_{i}}\taulim up$.
        \item Let $(p_{\lambda})_{\lambda\in\Lambda}$ and $(g_{\lambda})_{\lambda\in\Lambda}$ be nets in $\beta T$ and suppose that $p_{\lambda}g_{\lambda} \to r$ and $p_{\lambda}\to p$. Then, $(g_{\lambda})_{\lambda\in \Lambda}$ has a subnet $(g_{\lambda_{i}})_{i\in I}$ such that $g_{\lambda_{i}}\taulim up^{-1}r$.
    \end{enumerate}
\end{proposition}

From now on, to lighten the notation, whenever we need to pass to a subnet, we will keep the notation of the original net, and the notation $g_{\lambda}\taulim g$ will implicitly mean that we may pass to a subnet to obtain the limit.

For a $\uptau$-closed subgroup $F$ of $G$ the {\em derived group} $H(F)$ is given by:\begin{align*}
    H(F)=\bigcap\{\cltau{V}:V\text{ a } \uptau\text{-open neighborhood of } u \text{ in } F\}.
\end{align*}

The group $H(F)$ is a $\uptau$-closed normal subgroup and is the smallest $\uptau$-closed subgroup $H$ of $F$ such that $F/H$ is a compact Hausdorff topological group (for the quotient topology induced by the $\uptau$-topology).

For a point $x_{0}\in uX$, the {\em Ellis group} of the pointed system $(X,x_{0})$ is the $\uptau$-closed subgroup \begin{align*}
    \GG(X,x_{0})=\{g\in G:gx_{0}=x_{0}\}.
\end{align*}

When there is no ambiguity, we omit the fixed point $x_{0}$ and denote the Ellis group of $X$ by $\GG(X)$. Whenever $x_0$ is used, we will refer to the element with respect to which the group $\GG(X)$ is defined. For $\pi:(X,T)\to (Y,T)$ an extension of minimal systems we have $\GG(X)\subseteq \GG(Y)$, where the group $\GG(Y)$ is defined by the point $\pi(x_{0})$. The Ellis group is very useful when studying extensions. Below, there are a few applications of this concept. 

\begin{theorem}[see {\cite[Chapters 10 and 14]{Auslander_minimal_flows_and_extensions:1988}}]\label{thm: extensions_ellis_group}
    Let $\pi:(X, T) \to (Y, T)$ be an extension of minimal systems. Then,\begin{enumerate}[label=(\arabic*)]
        \item $\pi$ is proximal if and only if $\GG(X)= \GG(Y)$.
        \item $\pi$ is distal if and only if $\pi^{-1}(py_{0})=p\GG(Y)x_{0}$, for all $p\in M$.
        \item $\pi$ is equicontinuous if and only if $\pi$ is distal and $H(\GG(Y))\subseteq \GG(X)$.
        \item If $\pi$ is open, then $p\circ \pi^{-1}(y)=\pi^{-1}(py)$ for every $p\in M$ and $y\in Y$.
    \end{enumerate} 
\end{theorem}

\begin{proposition}[see {\cite[Chapter V]{deVries_elements_topological_dynamics:1993}}]\label{prop: pi_u_homeo_if and only if_proximal}
    Let $\pi\colon (X,T)\to (Y,T)$ an extension of minimal systems, and set $\pi_{u}=\pi|_{uX}$. Then, \begin{enumerate}[label=(\arabic*)]
        \item $\pi_{u}:uX\to uY$ is $\uptau$-continuous and $\uptau$-closed.
        \item $\pi_{u}:uX\to uY$ is a $\uptau$-homeomorphism if and only if $\pi$ is a proximal extension.
        \item $\pi_{u}:uX\to uY$ is $\uptau$-open.
    \end{enumerate}
\end{proposition}

Let $\pi:(X, T) \to (Y, T)$ be an extension of minimal systems. We say that $\pi$ is {\em RIC (relatively incontractible)} if for every $p\in M$\begin{align*}
    \pi^{-1}(py_{0}) = p\circ u\pi^{-1}(y_{0})=p \circ Fx_{0},
\end{align*}

where $F=\GG(Y)$ and $y_{0}=\pi(x_{0})$. We will sometimes use the following characterization of a RIC extension proved in \cite[Theorem A.2]{Akin_Glasner_Huang_Shao_Ye_conditions_transitive_chaotic:2010}: an extension $\pi$ is RIC if and only if it is open and for every $n\geq 1$ the minimal points are dense in the relation\begin{align*}
    R_{\pi}^{n} =\{(x_{1},\dots,x_{n})\in X^{n}: \pi(x_{i})=\pi(x_{j}), \forall 1\leq i \leq j\leq n\}.
\end{align*} 

Note that every distal extension of minimal systems is RIC, and every distal extension of minimal systems is open. The following are relative versions of weakly mixing.

We say that $\pi$ is \begin{enumerate}[label=(\arabic*)]
    \item {\em weakly mixing} if $(R_{\pi},T)$ is transitive.
    \item {\em totally weakly mixing} if $(R_{\pi}^{n},T)$ is transitive for every $n\geq 2$.
\end{enumerate}

Note that, for RIC extensions, weak mixing is equivalent to total weak mixing.

\section{Algebraic characterization of dynamical cubes} \label{sec:algebric_cubes}

 This section is devoted to studying the Host-Kra cube groups associated with subgroups of the universal minimal system. We start with generalities and show, among other things, the characterization of proximal extensions of systems of order $d$ announced in the introduction. 

 Recall that throughout this paper, $u$ denotes the idempotent of $M$ fixed in \cref{sec:univ_system} to define the group $G$. 

The proof of the next lemma is similar to that of {\cite[Proposition 1.55]{Glasner_ergodic_theory_joinings:2003}}.

\begin{lemma}
    Let $(X,T)$ be a minimal topological dynamical system, $d\geq 1$ be an integer and $\epsilon\subseteq [d]$. Let $u^{[d]} = (u,\dots,u)\in\bQ^{[d]}(M)$. Then $u^{[d]}$ coincides with a minimal idempotent of $\beta \mathcal{HK}^{[d]}(T)$ on $\bQ^{[d]}(X)$. \label{lem:min_idemp_cubo}
\end{lemma}
\begin{proof} 
Consider $\pi_{\epsilon}:\bQ^{[d]}(X)\to X_{\epsilon}=X$ the projection of $\bQ^{[d]}(X)$ on the $\epsilon$ coordinate. We consider the action of the group $\mathcal{HK}^{[d]}(T)$ on the $\epsilon$ coordinate via the projection $\pi_{\epsilon}$. Note that $\beta T$ coincides with $\beta \mathcal{HK}^{[d]}(T)$ on $X_{\epsilon}$, that is, for any $p\in \beta T$ there is a $\bold{q}\in \beta \mathcal{HK}^{[d]}(T)$ such that $px=\bold{q}x$ for all $x\in X_{\epsilon}$.

 It follows from the fact that $(\bQ^{[d]}(M),\mathcal{HK}^{[d]}(T))$ is minimal that there is a minimal idempotent $\bold{v}$ in $\beta \mathcal{HK}^{[d]}(T)$ such that $\bold{v}u^{[d]}=u^{[d]}$. Consider $\bold{v}_{\epsilon}\in \beta T$ such that $\bold{v}_{\epsilon}\pi_{\epsilon}(\bold{p}) = \pi_{\epsilon}(\bold{vp})$ for all $\bold{p}\in\bQ^{[d]}(M)$. Observe that $\bold{v}_{\epsilon}$ is a minimal idempotent on $M$. Let $u'\in J(\beta T v_{\epsilon})$ be a minimal idempotent equivalent to $u$, then\begin{align*}
u'=uu'=\bold{v}_{\epsilon}uu'=\bold{v}_{\epsilon}u'=\bold{v}_{\epsilon},
\end{align*}

and it follows that $u^{[d]}\bold{v}=\bold{v}$ on $\bQ^{[d]}(M)$. Thus, we have $u^{[d]}\in \beta \mathcal{HK}^{[d]}(T)\bold{v}$ on $\bQ^{[d]}(M)$. Since $\beta \mathcal{HK}^{[d]}(T)\bold{v}$ is a minimal left ideal, we conclude that $u^{[d]}$ coincide with a minimal idempotent of $\beta \mathcal{HK}^{[d]}(T)$ on $\bQ^{[d]}(M)$. Therefore, $u^{[d]}$ coincides with a minimal idempotent on $\bQ^{[d]}(X)$.
\end{proof}
\begin{remark}
One can alternatively  prove  \cref{lem:min_idemp_cubo} using the proof of \cite[Proposition 4.8]{Shao_Ye_regionally_prox_orderd:2012} and the fact that minimal idempotents in $E(\bQ^{[d]}(M))$ can be lifted to minimal idempotents in $\beta \mathcal{HK}^{[d]}(T)$.
\end{remark}

 Now, we define $M(\mathcal{HK}^{[d]}(T))=\beta \mathcal{HK}^{[d]}(T)u^{[d]}$ as the minimal left ideal of $\beta\mathcal{HK}^{[d]}(T)$ that contains $u^{[d]}$, and $G(\mathcal{HK}^{[d]}(T)) = u^{[d]}\beta \mathcal{HK}^{[d]}(T)u^{[d]}$ as the group with identity $u^{[d]}$. Observe that $G(\mathcal{HK}^{[d]}(T))$ coincides with $G$ on $X_{\epsilon}$ for any $\epsilon \subseteq [d]$ and $(X, T)$ minimal system.

\begin{theorem}\label{thm: dyn_cube_equal_HK_cube}
    Let $(X, T)$ be a minimal topological dynamical system and $d\geq 1$ be an integer. Then, $\mathcal{HK}^{[d]}(G)\bQ^{[d]}(X) \subseteq u^{[d]}\bQ^{[d]}(X)$. Moreover, for each $\bx \in \bQ^{[d]}(X)$, we have\begin{align*}
        \cltau{\HK{d}}\bx = u^{[d]}\bQ^{[d]}(X),
    \end{align*}

    where the closure is with respect to the $\uptau$-topology in $G(\mathcal{HK}^{[d]}(T))$.
\end{theorem}

\begin{proof}

    Let $g\in G$, $\bx\in\bQ^{[d]}(X)$ and $\alpha$ be a facet of $\{0,1\}^{d}$. Let $(t_{\lambda})_{\lambda\in\Lambda}\subseteq T$ be a net such that $t_{\lambda} \to g$. Then, we have $t_{\lambda}^{(\alpha)} \to g^{(\alpha)}$. Let $(\bs_{i})_{i\in I}\subseteq \mathcal{HK}^{[d]}(T)$ be such that $\bs_{i}x_{0}^{[d]} \to \bx$. Since $\mathcal{HK}^{[d]}(T)$ is a group we get that $(t_{\lambda}^{(\alpha)}\bs_{i})_{i\in I}\subseteq \mathcal{HK}^{[d]}(T)$, and then we have $(t_{\lambda}^{(\alpha)}\bx)_{\lambda\in\Lambda}\subseteq \bQ^{[d]}(X)$. Therefore, we deduce that $g^{(\alpha)}\bx \in u^{[d]}\bQ^{[d]}(X)$.  The set $\{ \boldsymbol{g} \in \mathcal{HK}^{[d]}(G): \boldsymbol{g}\bQ^{[d]}(X) \subseteq u^{[d]}\bQ^{[d]}(X)\}$  is clearly a subgroup of $\mathcal{HK}^{[d]}(G)$ and it contains the facets. Hence, it is all of $\mathcal{HK}^{[d]}(G)$. We deduce $\mathcal{HK}^{[d]}(G)\bQ^{[d]}(X) \subseteq u^{[d]}\bQ^{[d]}(X)$.

    Let $\bx\in \bQ^{[d]}(X)$. Then, there is $(\bt_{\lambda})_{\lambda\in\Lambda} \subseteq \mathcal{HK}^{[d]}(T)$ a net such that $\bt_{\lambda}x_{0}^{[d]}$ converge to $\bx$ (recall whenever $x_0$ is used, we will refer to the element with respect to which the group $\GG(X)$ is defined). By \cref{prop: parametrization_cubes}, there is $\bold{h}\in T^{[d]}$ such that \begin{align*}
        \bt_{\lambda} = \prod_{i=1}^{2^{d}}\bold{h}_{i}^{(\alpha_{i})},
    \end{align*}

    where $\alpha_{i}$ are the faces of the \cref{prop: parametrization_cubes}. Therefore, it follows that \begin{align*}
        u^{[d]}\bt_{\lambda} = \prod_{i}^{2^{d}}(u\bold{h}_{i})^{(\alpha_{i})}.
    \end{align*}

    Note that,\begin{align*}
        u^{[d]}\bt_{\lambda}x_{0}^{[d]} = \bt_{\lambda}u^{[d]}x_{0}^{[d]} \to \bx.
    \end{align*}

    Thus, we get that $u^{[d]}\bt_{\lambda}x_{0}^{[d]} \taulim u^{[d]}\bx$. Since $\prod_{i}^{2^{d}}(u\bold{h}_{i})^{(\alpha_{i})}$ is an element of $\HK{d}$, we conclude that $u^{[d]}\bx \in \cltau{\HK{d}}x_{0}^{[d]}$.
\end{proof}

Note that this proof of the theorem also allows us to extend it to general group actions, which are not necessarily abelian, for proximal extensions of distal systems. Also, when $X=M$ and $x_{0}=u$, it can be shown that $\cltau{\HK{d}}=G(\mathcal{HK}^{[d]}(T))$. Moreover, we have the following corollary:\begin{corollary}\label{cor: universal_hk_cube}
    Let $d\geq 1$ be an integer. Then, $\bQ^{[d]}(M)=M(\mathcal{HK}^{[d]}(T))$.
\end{corollary}

\begin{proof} 
It is clear that $\bQ^{[d]}(M) \subseteq M(\mathcal{HK}^{[d]}(T))$. Conversely, we have \begin{align*}
        M(\mathcal{HK}^{[d]}(T)) =\bigcup_{\bold{v}\in J(M(\mathcal{HK}^{[d]}(T)))}\bold{v}G(\mathcal{HK}^{[d]}(T)).
    \end{align*}
    
    Note that $G(\mathcal{HK}^{[d]}(T))\subseteq \cltau{\HK{d}}$. Indeed, let $\mathbf{g}\in G(\mathcal{HK}^{[d]}(T))$. Then there exists a net $(\bt_{\lambda})_{\lambda\in\Lambda} \subseteq \mathcal{HK}^{[d]}(T)$ such that $\bt_{\lambda} \to \mathbf{g}$. Thus, it follows that $\bt_{\lambda}u^{[d]} \taulim \mathbf{g}$, and hence $\mathbf{g} \in \cltau{\HK{d}}$.
    
    Therefore, by \cref{thm: dyn_cube_equal_HK_cube}, we obtain that \begin{align*}
        \bigcup_{\bold{v}\in J(M(\mathcal{HK}^{[d]}(T)))}\bold{v}\cltau{\HK{d}}=\bigcup_{\bold{v}\in J(M(\mathcal{HK}^{[d]}(T)))}\bold{v} \bQ^{[d]}(M)=\bQ^{[d]}(M).
    \end{align*}

    Hence, we conclude that $M(\mathcal{HK}^{[d]}(T)) \subseteq \bQ^{[d]}(M)$.
\end{proof}

The following lemma was implicitly proved in \cref{thm: dyn_cube_equal_HK_cube}, (using that $\bQ^{[d]}(M)$ is invariant under face transformations and the diagonal) but we state it here for convenience for the subsequent sections.

\begin{lemma}\label{lemma: HK_Ellis_contenido_cubo}
    Let $(X, T)$ be a minimal topological dynamical system and $d\geq 1$ be an integer. Consider $p_{0},\dots,p_{d}\in M$ and $\bx\in \bQ^{[d]}(X)$. Define $\bold{q} = (\bq_{\epsilon}:\epsilon\subseteq [d])$, where $\bq_{\epsilon} = p_{n_{1}}\dots p_{n_{k}}p_{0}$ for $\epsilon = \{n_{1},\dots,n_{k}\}$ and $\bq_{\emptyset} = p_{0}$. Then, $\bold{q}\in \bQ^{[d]}(M)$ and $\bold{q}\bx\in \bQ^{[d]}(X)$.
\end{lemma}

Using the algebraic characterization of the dynamical cubes, we can show an algebraic characterization of proximal extensions of equicontinuous systems.

\begin{proposition}\label{thm: equi_ext_if and only if_HG_abelian}
    Let $(X, T)$ be a minimal topological dynamical system. Then, $(X, T)$ is a proximal extension of its maximal equicontinuous factor if and only if $H(G)$ is trivial on $X$.
\end{proposition}

\begin{proof}
    Suppose $(X, T)$ is an equicontinuous system. By \cref{thm: extensions_ellis_group}, we get that $H(G) \subseteq \GG(X)$. Since every element $g\in G$ is continuous in $X$, we conclude that $H(G)$ is trivial on $X$. By \cref{prop: pi_u_homeo_if and only if_proximal}, we deduce that $H(G)$ is trivial on proximal extensions of an equicontinuous system.

Conversely, suppose that $H(G)$ is trivial on $X$. Let $(x, y)\in \RP(X)$. By \cref{thm: dyn_cube_equal_HK_cube}, there exist $(g_{0}^{\lambda})_{\lambda\in\Lambda},(g_{1}^{\lambda})_{\lambda\in\Lambda},(g_{2}^{\lambda})_{\lambda\in\Lambda} \subseteq G$ nets such that\begin{align*}
    g_{0}^{\lambda}x &\taulim ux\\
    g_{1}^{\lambda}g_{0}^{\lambda}x &\taulim ux\\
    g_{2}^{\lambda}g_{0}^{\lambda}x &\taulim ux\\
    g_{2}^{\lambda}g_{1}^{\lambda}g_{0}^{\lambda}x &\taulim uy.
\end{align*}

Let $p_{0},p_{1},p_{2}\in M$ be such that $g_{i}^{\lambda} \to p_{i}$ for each $i=0,1,2$. By \cref{lemma: conv_prod_tau_top}, we have $g_{i}^{\lambda} \taulim up_{i}$ for each $i=0,1,2$. Then, as $G/H(G)$ is a compact Hausdorff topological group for the quotient topology induced by the $\uptau$-topology, we obtain
\begin{align*}
    up_{0}H(G)x = up_{1}H(G)x= up_{2}H(G) = H(G)x = ux.
\end{align*}
where in the last equality, we use that $H(G)$ is trivial on $X$. 

Therefore, we get that\begin{align*}
    uy=H(G)y=up_{0}p_{1}p_{2}H(G)x = ux.
\end{align*}

Hence, we conclude that $\RP(X)=\P(X)$.
\end{proof}

Let $(X, T)$ be a minimal topological dynamical system and $d\geq 1$ be an integer. We say that $\HK{d}$ has the {\em property of unique $\uptau$-closure on $X$} if, for any $\bx\in \bQ^{[d]}(X)$, two elements of $\cltau{\HK{d}}\bx$ are equal whenever they have $2^{d}-1$ coordinates in common. By \cref{thm: dyn_cube_equal_HK_cube} and \cref{thm: shao_ye_thm}, we conclude the following theorem:

\begin{theorem}\label{thm: unique_tauclosure_iff_prox_ext_nil}
    Let $(X, T)$ be a minimal topological dynamical system and $d\geq 1$ be an integer. Then, $\HK{d+1}$ has the property of unique $\uptau$-closure on $X$ if and only if $(X, T)$ is a proximal extension of its maximal factor of order $d$.
\end{theorem}

\begin{proof}
    Suppose that $(X,T)$ is a system of order $d$. Then, by \cref{thm: shao_ye_thm} and \cref{thm: dyn_cube_equal_HK_cube}, we have that $\HK{d+1}$ has the property of unique $\uptau$-closure on $X$. Hence, it follows from \cref{prop: pi_u_homeo_if and only if_proximal} that any proximal extension of $(X,T)$ also satisfies the unique $\uptau$-closure property of $\HK{d+1}$.
    
    Conversely, suppose that $\HK{d+1}$ has the property of unique $\uptau$-closure on $X$. Let $(x,y)\in \RP^{[d]}(X)$. In particular, by \cref{thm: shao_ye_thm}, we have that $(x,y,\dots,y)\in\bQ^{[d+1]}(X)$. Therefore, by \cref{thm: dyn_cube_equal_HK_cube}, it follows that $(ux,uy,\dots,uy)\in \cltau{\HK{d+1}}x^{[d+1]}$. Since both $u^{[d+1]}x^{[d+1]}$ and $(ux,uy,\dots,uy)$ are elements of $\cltau{\HK{d+1}}x^{[d+1]}$, the property of unique $\uptau$-closure of $\HK{d+1}$ on $X$ implies that $u^{[d+1]}x^{[d+1]}=(ux,uy,\dots,uy)$. Thus, we obtain that $ux=uy$, which shows that $\RP^{[d]}(X)= \P(X)$. Hence, $(X,T)$ is a proximal extension of its maximal factor of order $d$.
\end{proof}

\begin{definition} \label{def:tau_commutator}
    Let $T$ be a topological group, $M$ be a minimal ideal of $\beta T$ and $u\in J(M)$. Let $G=uM$. The $\uptau$-topological commutator subgroups $\Gtau{j}$, $j\geq 1$, are defined by setting $\Gtau{1}=G$ and $\Gtau{j+1} = \cltau{[\Gtau{j},G]}$.
\end{definition}

This definition corresponds to the topological commutators when considering the $\uptau$-topology.

\begin{lemma}\label{lemma: Gdtau_normal_subgroup}
For each integer $d\geq 1$, the set  $\Gtau{d}$ is a normal subgroup of $G$. 
\end{lemma}

\begin{proof}
    We proceed by induction on $d$. For $d=1$ the statement is obvious. Let $d\geq 1$ be an integer, and suppose the statement is true for $d$.

    Let $g,h\in \Gtau{d+1}$. Then there are nets $(g_{\lambda})_{\lambda\in\Lambda},(h_{\lambda})_{\lambda\in\Lambda}$  in $[\Gtau{d},G]$ such that\begin{align*}
        g_{\lambda} &\taulim g\\
        h_{\lambda} &\taulim h.
    \end{align*}

      Let $p\in [\Gtau{d},G]$. Note that $h_{\lambda}p\to hp$, and so we obtain that $hp \in \Gtau{d+1}$ for every $p\in [\Gtau{d},G]$. Since $\Gtau{d}\unlhd G$ implies $[\Gtau{d},G]\unlhd G$, we get that $g_{\lambda}[g_{\lambda},h]\in [\Gtau{d},G]$.  Therefore, it follows that $
        g_{\lambda}h = hg_{\lambda}[g_{\lambda},h] \in \Gtau{d+1}$. As  $ g_{\lambda}h \taulim gh$, we deduce that $gh\in\Gtau{d+1}$.

    Since the inversion is continuous with respect to the $\uptau$-topology, we get that $\Gtau{d+1}$ is closed under inversion. Since multiplication is separately continuous and since $[\Gtau{d},G]$ is a normal subgroup of $G$, we get that $\Gtau{d+1}$ is a normal subgroup of $G$ too.
\end{proof}

 Note that by \cref{lemma: Gdtau_normal_subgroup}, $\Gtau{j+1}$ is the smallest $\uptau$-closed subgroup that contains $[\Gtau{j},G]$.

\begin{lemma}\label{lemma: HK_invaraint_g_alpha}
    Let $(X, T)$ be a minimal topological dynamical system and $d\geq 1$ be an integer. Then, for every $j\in\N$ with $1\leq j\leq d$ and $g\in G_{j}^{\uptau\mathrm{-top}}$, we have that $\cltau{\HK{d}}$ is invariant under $g^{(\alpha)}$ for every face $\alpha$ of codimension $j$. 
\end{lemma}
\begin{proof}
    Let $g\in \Gtau{1}=G$ and $\alpha$ be a face of codimension $1$. By the definition of $\HK{d}$, we have $g^{(\alpha)}\in \HK{d}$, then we deduce that $\HK{d}$ is invariant under $g^{(\alpha)}$.

    Consider $\bh\in \cltau{\HK{d}}$ and $(\bh_{\lambda})_{\lambda\in\Lambda}$ be a net in $\HK{d}$ such that $\bh_{\lambda}\taulim \bh$. Since $\HK{d}$ is invariant under $g^{(\alpha)}$, we get that\begin{align*}
       \HK{d} \ni g^{(\alpha)}\bh_{\lambda} \taulim g^{(\alpha)}\bh.
    \end{align*}

    Therefore, we conclude that $\cltau{\HK{d}}$ is invariant under $g^{(\alpha)}$.

    Let $2\leq j \leq d$ and suppose that the statement is true for every $i<j$. Let $\alpha$ be a face of codimension $j$. We can see $\alpha$ as the intersection of a face $\beta$ of codimension $j-1$ and a face $\gamma$ of codimension $1$. Note that\begin{align*}
        [g,h]^{(\alpha)} = [g^{(\gamma)},h^{(\beta)}].
    \end{align*}

    Since $\cltau{\HK{d}}$ is invariant under $g^{(\gamma)},(g^{-1})^{(\gamma)}, h^{(\beta)}$ and $(\beta^{-1})^{(\beta)}$, we get that $[g,h]^{(\alpha)}$ leaves $\cltau{\HK{d}}$ invariant. Then, we have that $\cltau{\HK{d}}$ is invariant under $r^{(\alpha)}$ for each $r\in G_{j}$. Since $\cltau{\HK{d}}$ is $\uptau$-closed, we conclude that $r^{(\alpha)}$ leaves invariant $\cltau{\HK{d}}$ for each $r\in \Gtau{j}$.
\end{proof}

Let $(X, T)$ be a minimal topological dynamical system and $d\geq 1$ be an integer. We say $G$ is {\em $d$-step $\uptau$-topologically nilpotent} on $X$ if $\Gtau{d+1}x=\{ux\}$ for all $x\in X$.

\begin{proposition}\label{prop: unique_tauclosure_then_G_nil}
    Let $(X, T)$ be a minimal topological dynamical system and $d\geq 0$ be an integer such that $\HK{d+1}$ has the property of unique $\uptau$-closure on $X$. Then, $G$ is $d$-step $\uptau$-topologically nilpotent on $X$.
\end{proposition}

\begin{proof}
    Let $g \in \Gtau{d+1}$ and $\alpha=\{\overrightarrow{1}\}$. By \cref{lemma: HK_invaraint_g_alpha}, we get that $g^{(\alpha)}u^{[d+1]}\in\cltau{\HK{d+1}}$. Note that $g^{(\alpha)}u^{[d+1]}$ and $u^{[d+1]}$ have $2^{d+1}-1$ coordinates in common, so we have $gx=ux$ for each $x\in X$. That is, $G$ is $d$-step $\uptau$-topologically nilpotent on $X$.
\end{proof}

By \cref{thm: unique_tauclosure_iff_prox_ext_nil} and \cref{prop: unique_tauclosure_then_G_nil}, we conclude the following result.

\begin{theorem}\label{thm: G_tau_top_nil_proximal_ext_nil}
    Let $(X, T)$ be a minimal topological dynamical system such that it is a proximal extension of its maximal factor of order $d$. Then, $G$ is $d$-step $\uptau$-topologically nilpotent on $X$.
\end{theorem}
This theorem raises the following question:

\begin{question} \label{ques:G-nil}
    If $G$ is $d$-step $\uptau$-topologically nilpotent on a minimal system $(X, T)$, is $(X, T)$ a proximal extension of a system of order $d$?
\end{question}
We suspect that the answer to \cref{ques:G-nil} is negative, but currently we do not have a counterexample, even when $G$ is abelian on $X$. Note that the proximal relation of a proximal extension of a system of order $d$ is an equivalence relation (because in such a case $\P=\RP^{[d]}$).  We still do not know whether the action of $G$ being abelian on $X$ implies that $\P$ is an equivalence relation.

\begin{remark}
    Using \cref{thm: G_tau_top_nil_proximal_ext_nil} it is straightforward to prove that the virtual automorphism group, as defined in \cite{Auslander_Glasner_virtual_automorphism:2021}, of a proximal extension of a system of order $d$ is a $d$-step nilpotent group. This generalizes what was proven in \cite[Theorem 5.5
]{Donoso_Durand_Maass_Petite_automorphism_low_complexity:2016} for the automorphism group of a system of order $d$ (note that in the minimal distal case, the automorphism group and the virtual automorphism group coincide by  \cite[Theorem 5.1]{Auslander_Glasner_virtual_automorphism:2021}).
\end{remark}

In \cite{Qiu_Zhao_topnilpotent_enveloping_nil:2022}, it was proven that if a minimal system has a topologically nilpotent enveloping semigroup, then the enveloping semigroup of its dynamical cube also has the same property. The following proposition generalizes their result.

\begin{proposition}\label{thm: G_nil_iff_G_nil_in_cube}
    Let $(X, T)$ be a minimal topological dynamical system and $d\geq 1$ be an integer. If $G$ is a group $d$-step $\uptau$-topologically nilpotent on $X$, then $G(\mathcal{HK}^{[d]}(T))$ is a group $d$-step $\uptau$-topologically nilpotent on $\bQ^{[d]}(X)$.
\end{proposition}

\begin{proof}
    Let $\epsilon\subseteq [d]$ and $\bx\in\bQ^{[d]}(X)$. Then,\begin{align*}
        G_{2}^{\uptau\mathrm{-top}}(\mathcal{HK}^{[d]}(T))\bx_{\epsilon}&=\mathrm{cl}_{\uptau}^{[d]}([G(\mathcal{HK}^{[d]}(T)),G(\mathcal{HK}^{[d]}(T))])\bx_{\epsilon}\\
        &= \pi_{\epsilon}(u^{[d]}(u^{[d]}\circ [G(\mathcal{HK}^{[d]}(T)),G(\mathcal{HK}^{[d]}(T))]))\bx_{\epsilon}\\
        &= u(u\circ \pi_{\epsilon}([G(\mathcal{HK}^{[d]}(T)),G(\mathcal{HK}^{[d]}(T))])\bx_{\epsilon}\\
        &= u(u \circ [G(T),G(T)])x_{\epsilon}\\
        &= G_{2}^{\uptau\mathrm{-top}}(T)\bx_{\epsilon},
    \end{align*}
    
    where $\mathrm{cl}_{\uptau}^{[d]}$ is the closure with respect to the $\uptau$-topology in $G(\mathcal{HK}^{[d]}(T))$ and $\pi_{\epsilon}: M^{[d]} \to M$ is the projection onto the $\epsilon$-coordinate.  Now, let $d\geq 1$ and suppose that $G_{d}^{\uptau\mathrm{-top}}(\mathcal{HK}^{[d]}(T))$ coincides with $G_{d}^{\uptau\mathrm{-top}}$ on $X_{\epsilon}$. As before, we have\begin{align*}
        G_{d+1}^{\uptau\mathrm{-top}}(\mathcal{HK}^{[d]}(T))\bx_{\epsilon}&=\mathrm{cl}_{\uptau}^{[d]}([G_{d}^{\uptau\mathrm{-top}}(\mathcal{HK}^{[d]}(T)),G(\mathcal{HK}^{[d]}(T))])\bx_{\epsilon}\\
        &= u(u\circ \pi_{\epsilon}([G_{d}^{\uptau\mathrm{-top}}(\mathcal{HK}^{[d]}(T)),G(\mathcal{HK}^{[d]}(T))])\bx_{\epsilon}\\
        &= u(u \circ [G_{d}^{\uptau\mathrm{-top}}(T),G(T)])\bx_{\epsilon}\\
        &= G_{d+1}^{\uptau\mathrm{-top}}(T)\bx_{\epsilon}.
    \end{align*}    
    
    Thus, by induction, we obtain that $G_{d}^{\uptau\mathrm{-top}}(\mathcal{HK}^{[d]}(T))$ coincide with $G_{d}^{\uptau\mathrm{-top}}$ on $X_{\epsilon}$, for each $d\geq 1$ being an integer. Therefore, if $G$ is a group $d$-step $\uptau$-topologically nilpotent on $X$, then $G(\mathcal{HK}^{[d]}(T))$ is a group $d$-step $\uptau$-topologically nilpotent on $\bQ^{[d]}(X)$.
\end{proof}

\section{Regionally proximal relation for distal systems and their proximal extensions} \label{sec:Regionally_proximal}
In this section, we provide an algebraic characterization of the regionally proximal relation for distal systems and their proximal extensions.  Next, we extend this characterization to any minimal system. We begin by outlining Qiu and Zhao's proof of the algebraic characterization of $\RP^{[j]}$ for systems of order $d$, where $d$ and $j$ are integers with $d, j \geq 1$, extending their proof to proximal extensions of distal systems using the Galois theory of minimal systems.

\begin{definition}
    Let $(X, T)$ be a minimal topological dynamical system and let $d\geq 1$ be an integer. We define\begin{align*}
        R_{d}^{\uptau}(X)= \{(x,vgx):x\in X, M'\subseteq \beta T \text{ minimal left ideal}, v,u'\in J(M')\text{ s.t }u\sim u', g\in G(u',M')_{d+1}^{\uptau\mathrm{-top}} \}.
    \end{align*}
where for a left minimal ideal $M'$ and $v\in J(M')$,  $G(v, M')\coloneqq vM'$.
\end{definition}

We will see later that this relation generalizes the definition given by Qiu and Zhao for systems of order $d$ in \cite{Qiu_Zhao_topnilpotent_enveloping_nil:2022}.  Note that if $G$ is a group $d$-step $\uptau$-topologically nilpotent on a minimal system $(X,T)$, we have $R_{d}^{\uptau}(X)=\P(X)$. Consequently, if $R_{d}^{\uptau}(X)=\RP^{[d]}(X)$, then \cref{ques:G-nil} has a positive answer. It is therefore useful to identify sufficient conditions for which $R_{d}^{\uptau}(X)=\RP^{[d]}(X)$. Later, we will show that a sufficient condition is that $X$ is a proximal extension of a distal system.

Observe that\begin{align*}
    \P(X) \subseteq \cdots \subseteq R_{d+1}^{\uptau}(X) \subseteq R_{d}^{\uptau}(X) \subseteq \cdots \subseteq R_{1}^{\uptau}(X).
\end{align*}

Note that $R^{*}_{d}(X)$, the smallest closed equivalence relation which contains $R_{d}^{\uptau}(X)$, for a minimal system $(X,T)$ and $d\geq 1$ an integer satisfies the following property: $X/R^{*}_{d}(X)$ is the maximal distal factor of $X$ such that $\Gtau{d+1}$ is trivial on it. Moreover, the relation $R_{d}^{*}$ coincides with the regionally proximal relation of order $d$, as can be seen below.

 \begin{proposition}\label{prop: Rdtau_subset_RPd}
     Let $(X, S)$ be a minimal topological dynamical system and $d\geq 1$ be an integer. Then, $R_{d}^{\uptau}(X)\subseteq \RP^{[d]}(X)$. Moreover, $R_{d}^{*}(X)=\RP^{[d]}(X)$.
 \end{proposition}

 \begin{proof}
     Consider $\pi:X\to X/\RP^{[d]}(X)$ the quotient map. Let $(x,y)\in R_{d}^{\uptau}(X)$, note that $\pi\times\pi(x,y)\in R_{d}^{\uptau}(X/\RP^{[d]}(X))$. It follows from \cref{thm: G_tau_top_nil_proximal_ext_nil} that $R_{d}^{\uptau}(X/\RP^{[d]}(X))=\Delta$. Therefore, we have $\pi(x)=\pi(y)$, that is, $(x,y)\in \RP^{[d]}(X)$.

     Since $R_{d}^{\uptau}(X)\subseteq \RP^{[d]}(X)$, we get that $X/R_{d}^{*}(X)$ is an extension of $X/\RP^{[d]}(X)$. Note that $X/R_{d}^{*}(X)$ is a distal system such that its enveloping semigroup is a $d$-step topologically nilpotent, so we have that $X/R_{d}^{*}(X)$ is a system of order $d$. Therefore, by \cref{thm: nil_iff_enveloping_nil}, we conclude that $R_{d}^{*}(X)=\RP^{[d]}(X)$.
 \end{proof}

 By the following results, we can show that $R_{d}^{\uptau}(X)$ is a closed relation for proximal extensions of distal systems.

 \begin{theorem}[see {\cite[Chapter 4]{Auslander_minimal_flows_and_extensions:1988}}]\label{thm: Ellis_joint_cont}
     Let $T$ be a locally compact Hausdorff space with a group structure for which group multiplication is separately continuous, and suppose $T$ acts on the compact Hausdorff space $X$ in a separately continuous manner. Then, the action of $T$ on $X$ is jointly continuous.
 \end{theorem}

 \begin{lemma}
     Let $F\leq G$ be a $\uptau$-closed subgroup and $K\unlhd F$ such that $H(F)\leq K$. Then, $\cltau{K}\unlhd F$.
 \end{lemma}

 \begin{proof}
     Let $g,h\in \cltau{K}$. Let $(g_{\lambda})_{\lambda\in\Lambda}$ and $(h_{\lambda})_{\lambda\in\Lambda}$ be nets in $K$ such that $g_{\lambda} \taulim g$ and $h_{\lambda} \taulim h$.

     Since $F/H(F)$ is a Hausdorff topological group, therefore there is a $r\in H(F)$ such that\begin{align*}
         g_{\lambda}h_{\lambda} \taulim ghr.
     \end{align*}

     Note that $(g_{\lambda}h_{\lambda}r^{-1})_{\lambda\in\Lambda}$ converges to $gh$ with respect to $\uptau$-topology, so we have $gh\in \cltau{K}$.

    It follows from the facts that $K$ is a normal subgroup and that inversion is continuous and multiplication is separately continuous with respect to the $\uptau$-topology that $\cltau{K}$ is closed under inversion and is a normal subgroup of $G$.
 \end{proof}

\begin{lemma}\label{lemma: joint_cont_in_distal}
    Let $(X, T)$ be a distal minimal topological dynamical system and $d\geq 1$ be an integer. Let $(g_{\lambda})_{\lambda\in\Lambda} \subseteq G$ be a net, $(h_{\lambda})_{\lambda\in\Lambda} \subseteq G_{d}^{\uptau\mathrm{-top}}$ such that $g_{\lambda} \to p \in M$ and $h_{\lambda} \taulim h\in G$. Then, every limit point of $(h_{\lambda}g_{\lambda}x)_{\lambda\in\Lambda} \subseteq G$ belong to $hpG_{d}^{\uptau\mathrm{-top}}x$ for each $x\in X$.
\end{lemma}

\begin{proof}
    Let $x\in X$. Consider $H_{d} = \cltau{\langle H(\Gtau{d}),\Gtau{d+1} \rangle}$, it follows from the last lemma that $H_{d} \unlhd \Gtau{d}$. We define the following map:\begin{align*}
        \Gtau{d}/H_{d} \times G/H_{d} x &\to G/H_{d} x\\
        (hH_{d},gH_{d}x) &\mapsto hgH_{d}x,
    \end{align*}

    where $\Gtau{d}/H_{d}$ is endowed with the $\uptau$-topology and $G/H_{d} x$ is endowed with the usual topology. Since $H(\Gtau{d})\subseteq H_{d}$, we get that $\Gtau{d}/H_{d}$ is a Hausdorff compact group. It follows from $(X, T)$ being distal that $G/H_{d} x$ is a compact Hausdorff space. Note that for $g\in G$ and $h\in \Gtau{d}$ we have the following:\begin{align*}
        hgH_{d}x=ghH_{d}x.
    \end{align*}

    We will show that this map is separately continuous. Observe that this map is continuous for any translation by an element of $G/H_{d}x$. Let $(h_{\lambda})_{\lambda\in\Lambda}$ be a net in $\Gtau{d}$ and $h\in \Gtau{d}$ such that $h_{\lambda} \taulim h$. Consider $r\in M$ such that $h_{\lambda}\to r$, then we have $h_{\lambda} \taulim ur \in \Gtau{d}$. Since $\Gtau{d}/H_{d}$ is a Hausdorff space, we get that $hH_{d}=urH_{d}$. Therefore, for any $g\in G$ we have the following:\begin{align*}
        h_{\lambda}gH_{d}x \to rgH_{d}x = urgH_{d}x = grH_{d}x = ghH_{d}x = hgH_{d}x.
    \end{align*}

    Thus, this map is separately continuous. By \cref{thm: Ellis_joint_cont}, we deduce its continuity, completing the proof.
\end{proof}

\begin{theorem}\label{thm: Rdtau_closed}
    Let $(X, T)$ be a topological dynamical system such that it is a proximal extension of a distal system and $d\geq 1$ be an integer. Then, $R_{d}^{\uptau}(X)$ is an invariant equivalence relation and closed with respect to the usual topology of $X\times X$.
\end{theorem}

\begin{proof}
    Since $\P(X)$ is an equivalence relation, then\begin{align*}
        R_{d}^{\uptau}(X)=\{(x,vgx):v\in J(M),g\in G_{d+1}^{\uptau\mathrm{-top}}, x\in X\}.
    \end{align*}

    It is easy to see that $R_{d}^{\uptau}(X)$ is an invariant relation.

    Let $(x,vgx),(x,whx)\in R_{d}^{\uptau}(X)$, where $v,w\in J(M)$ and $g,h\in G_{d+1}^{\uptau\mathrm{-top}}$. By \cref{lemma: Gdtau_normal_subgroup} we get that $hg^{-1} \in \Gtau{d+1}$. Then, we have\begin{align*}
        (vgx,whx)=(vgx,w(hg^{-1})vgx) \in R_{d}^{\uptau}(X).
    \end{align*}

    Therefore, $R_{d}^{\uptau}(X)$ is an equivalence relation.

    Let $(x_{\lambda},v_{\lambda}g_{\lambda}x_{\lambda})_{\lambda\in\Lambda}\subseteq R_{d}^{\uptau}(X)$ be a net convergent to $(x,y)$. Let $(p_{\lambda})_{\lambda\in \Lambda}$ be a net in $M$ such that $x_{\lambda}=p_{\lambda}x$ for each $\lambda\in \Lambda$.

    Consider $v \in M$ such that $v_{\lambda} \to v$. Since $\P(X)$ is closed, there exists a $v_{z}\in J(M)$ for every $z\in X$ such that \begin{align*}
        vz=v_{z}z.
    \end{align*}

    Thus, we get that $vz\in J(M)z$. In particular, for any $z\in X$, we have $uz=uv^{-1}z=uvz$. Let $s\in M$ satisfy $v_{\lambda}g_{\lambda}\to s$. By \cref{lemma: conv_prod_tau_top}, we have \begin{align*}
        g_{\lambda}z \taulim uv^{-1}sz = usz.
    \end{align*}

    Consequently, $sz\in J(M)G_{d+1}^{\uptau\mathrm{-top}}z$ for each $z\in X$. 
    
    Consider $w\in J(M)$, $g\in G_{d+1}^{\uptau\mathrm{-top}}$ and $p,q\in M$ such that\begin{align*}
        p_{\lambda} &\to p\\
        v_{\lambda}g_{\lambda} &\to wg\\
        v_{\lambda}g_{\lambda}p_{\lambda} &\to q,
    \end{align*}

    where the limits are in $X^{X}$. Note that $g_{\lambda} \taulim g$ on $X$.

    Now, we will prove that $qx\in J(M)gG_{d+1}^{\uptau\mathrm{-top}}px$. Consider $\pi:X \to Y$ a proximal extension, where $Y$ is a minimal distal system. Note that\begin{align*}
        \pi(v_{\lambda}g_{\lambda}p_{\lambda}x) = g_{\lambda}p_{\lambda}\pi(x).
    \end{align*}
    
    It follows from \cref{lemma: joint_cont_in_distal} that $q\pi(x)\in gpG_{d+1}^{\uptau\mathrm{-top}}\pi(x)$. Since $\pi$ is a proximal extension, $qx\in J(M)gpG_{d+1}x$. By \cref{lemma: Gdtau_normal_subgroup}, $G_{d+1}^{\uptau\mathrm{-top}}$ is a normal subgroup, so $upG_{d+1}=G_{d+1}p$.

    Let $r\in G_{d+1}^{\uptau\mathrm{-top}}$ be such that $qx=wgrpx$. Then, \begin{align*}
        (x,y)=(x,wgrpx)=(x,wgrx)\in R_{d}^{\uptau}(X).
    \end{align*}

    Hence, we conclude that $R_{d}^{\uptau}(X)$ is closed.
\end{proof}

By \cref{prop: Rdtau_subset_RPd} and \cref{thm: Rdtau_closed}, we can conclude the following result.

\begin{corollary}\label{cor: RPd_is_Rdtau}
    Let $(X, S)$ be a topological dynamical system such that it is a proximal extension of a distal system and $d\geq 1$ be an integer. Then, $R_{d}^{\uptau}(X)=\RP^{[d]}(X)$.
\end{corollary}
\begin{remark}
If we did not know a priori that $\RP^{[d]}$ is an equivalence relation for all systems, we may still be able to deduce \cref{cor: RPd_is_Rdtau} if $\RP^{[d]}$ is known to be an equivalence relation for distal systems. 
\end{remark}

By \cref{cor: RPd_is_Rdtau}, we have the following natural question:

\begin{question}
    Let $(X, T)$ be a minimal topological dynamical system. Is it true that $R_{d}^{\uptau}(X)=\RP^{[d]}(X)$?
\end{question}

Now, we will show that $R_{d}^{\tau}$ coincide with the relation defined by Qiu and Zhao in \cite{Qiu_Zhao_topnilpotent_enveloping_nil:2022} for systems of order $j$, where $d,j\geq 1$ are integers.

\begin{lemma}\label{lemma: Gdtau_Gtop_coincide_nil}
    Let $(X, S)$ be a minimal system of order $d$. Then, $G_{k}^{\mathrm{top}}x=G_{k}^{\uptau\mathrm{-top}}x$ for each $x\in X$ and $k=1,\dots,d+1$.
\end{lemma}

\begin{proof}
    We show it by induction on $d$. Note that the statement is true for $d=1$.

    Let $d>1$ be an integer and $x\in X$; suppose the statement is true for each $j\leq d$. By \cref{thm: G_tau_top_nil_proximal_ext_nil}, we have $G_{d+1}^{\mathrm{top}}x = G_{d+1}^{\uptau-\mathrm{top}}x=\{x\}$, so it suffices to prove the statement for $k=1,\dots,d$. Let $g\in \Gtau{k}$, $k\in\{1,\dots, d\}$ and $\pi:X\to X/\RP^{[d-1]}(X)$ be the quotient map. Since $X/\RP^{[d-1]}(X)$ is a system of order $d-1$, then, by the induction hypothesis, there is a $g'\in G_{k}^{\mathrm{top}}$ such that\begin{align*}
        \pi(gx)=g\pi(x)=g'\pi(x)=\pi(g'x)
    \end{align*}

    Therefore, we get that $(gx,g'x)\in \RP^{[d-1]}(X)$. It follows from \cref{thm: charc_RPd_Rd} that there is a $r\in G_{d}^{\mathrm{top}}$ such that $gx=g'rx$. Thus, we have $G_{k}^{\mathrm{top}}x=\Gtau{k}x$. By \cref{thm: G_tau_top_nil_proximal_ext_nil}, we conclude that the statement is true for each $k=1,\dots,d+1$.
\end{proof}

\begin{proposition}
    Let $(X, S)$ be a minimal topological dynamical system, such that it is a proximal extension of a system of order $d$. Then,\begin{align*}
        \RP^{[j]}(X) = R_{j}^{\uptau}(X) =\{(x,vgx):v\in J(M), g\in G_{j+1}^{\mathrm{top}}\}.
    \end{align*}
    for each $j=1,\dots,d+1$. 
\end{proposition}

\begin{proof}
    Let $j=1,\dots, d+1$. By \cref{cor: RPd_is_Rdtau}, we have $\RP^{[j]}(X) = R_{j}^{\uptau}(X)$. Furthermore, it follows from \cref{lemma: Gdtau_Gtop_coincide_nil} that \begin{align*}
         R_{j}^{\uptau}(X) =\{(x,vgx):v\in J(M), g\in G_{j+1}^{\mathrm{top}}\},
    \end{align*}

    which completes the proof.
\end{proof}

We can extend the results of systems of order $d$ to systems of order $\infty$. Let $(X, T)$ be a topological dynamical system. A pair $(x,y)\in X\times X$ is said to be {\em regionally proximal of order $\infty$} if it is regionally proximal of order $d$ for all $d\geq 1$. The set of regionally proximal pairs of order $\infty$ is denoted by $\RP^{[\infty]}(X)$ and is called the {\em regionally proximal relation of order $\infty$}. We say that $(X, T)$ is a {\em system of order $\infty$} if $\RP^{[\infty]}(X)$ coincides with the diagonal relation.

Dong, Donoso, Maass, Shao and Ye studied this relation \cite{Dong_Donoso_Maass_Shao_Ye_infinite_step_nil:2013} and proved that a system of order $\infty$ is an inverse limit of nilsystems for $\Z$-actions. It is easy to prove that a system of order $\infty$ is an inverse limit of its maximal factor systems of order $d$, for $d\geq 1$.

\begin{lemma}\label{lemma: Gj_Gjtau_coincide_nil_infty}
    Let $(X, S)$ be a minimal system of order $\infty$. Then, $G_{j}^{\mathrm{top}}x=G_{j}^{\uptau\mathrm{-top}}x$ for each $x\in X$ and $j\geq 1$.
\end{lemma}

\begin{proof}
    Since $(X, S)$ is system of order $\infty$, then, without loss of generality, $X=\ilim{X/\RP^{[d]}(X)}$, that is, there are $\pi_{d}: X/\RP^{[d+1]}(X)\to X/\RP^{[d]}(X)$ extensions, such that\begin{align*}
        X=\left\lbrace (x_{d})_{d\geq 1} \in \prod_{d\geq 1}X/\RP^{[d]}(X): \pi_{d}(x_{d+1})=x_{d} \right\rbrace.
    \end{align*}

    Let $j\geq 1$ be an integer and let $g\in \Gtau{j}$. Note that for $(x_{d})_{d\geq 1} \in X$ we have $g(x_{d})_{d} = (gx_{d})_{d\geq 1}$. By \cref{thm: G_tau_top_nil_proximal_ext_nil}, we get that $\Gtau{j}x_{d}=x_{d}$ for each $d>j$. Moreover, by \cref{lemma: Gdtau_Gtop_coincide_nil}, there is a $g'\in G_{j}^{\mathrm{top}}$ such that $gx_{j}=g'x_{j}$. Then,\begin{align*}
        gx_{j-1}=\pi_{j-1}(gx_{j})=\pi_{j-1}(g'x_{j})=g'x_{j-1}.
    \end{align*}

Inductively, we can conclude that $g(x_d)_{d\geq 1} = g'(x_d)_{d\geq 1}$.\end{proof}

Therefore, by \cref{prop: pi_u_homeo_if and only if_proximal}, \cref{lemma: Gj_Gjtau_coincide_nil_infty} and \cref{cor: RPd_is_Rdtau}, we can conclude the following result.

\begin{proposition}\label{prop:RPd_Rd_nil_infty}
    Let $(X, S)$ be a minimal topological dynamical system such that it is a proximal extension of a system of order $\infty$. Then,\begin{align*}
        \RP^{[d]}(X)=\{(x,vgx):v\in J(M), g\in G_{d+1}^{\mathrm{top}}\}.
    \end{align*}
    for every $d\geq 1$. 
\end{proposition}

Using the enveloping semigroup, we show an algebraic characterization of systems of order $\infty$.

\begin{theorem}
    Let $(X, S)$ be a distal minimal topological dynamical system. Then, $(X, S)$ is a system of order $\infty$ if and only if $\bigcap_{d\in\N}E_{d}^{\mathrm{top}}(X, S)$ is trivial.
\end{theorem}

\begin{proof}
    Suppose that $(X, S)$ is a system of order $\infty$. Then,\begin{align*}
        \RP^{[\infty]}(X)=\bigcap_{d\geq 1}\RP^{[d]}(X)=\bigcap_{d\geq 1}R_{d}^{\uptau}(X)=\Delta.
    \end{align*}

    Let $p\in \bigcap_{d\geq 1}E_{d}^{\mathrm{top}}(X,S)$. By \cref{lemma: Gj_Gjtau_coincide_nil_infty} and distality of $(X,S)$, we get that $(x,px)\in R_{d}^{\uptau}(X)$ for each $x\in X$ and $d\geq 1$. Thus, we have $px=x$ for all $x\in X$. Therefore, we deduce that $\bigcap_{d\geq 1}E_{d}^{\mathrm{top}}(X, S)$ is trivial.

    Suppose that $\bigcap_{d\geq 1}E_{d}^{\mathrm{top}}(X, S)$ is trivial. Let $x\in X$ and let $(x,px) \in \RP^{[\infty]}(X)$, where $p\in E(X, S)$. In particular, we have $(x,px)\in R_{d}^{\uptau}(X)$ for each $d\geq 1$. Therefore, there is $g_{d}\in \Gtau{d+1}$ such that $px=g_{d}x$.

    Note that $\Gtau{d+1} = \bigcap_{j=1}^{d+1}\Gtau{j}$ and $\Gtau{d+1}$ is a closed set. Then, there is a $y\in \bigcap_{j\geq 1}\Gtau{j}x$ such that $g_{d}x\to x$. Since $\bigcap_{d\geq 1}E_{d}^{\mathrm{top}}(X, S)$ is trivial, we get that $y=x$. Therefore, we have\begin{align*}
        px=g_{d}x\to y=x.
    \end{align*}

    Hence, it follows that $px=x$, and we conclude that $\RP^{[\infty]}(X)=\Delta$.
\end{proof}

\section{Applications}

\subsection{$\RP^{[d]}$ is an equivalence relation} \label{sec:applic_RP}
In this subsection, we give a novel proof that $\RP^{[d]}$ is an equivalence relation for any minimal system establishing an algebraic characterization of regionally proximal relations for arbitrary minimal systems, not necessarily proximal extensions of distal systems. To prove this, we need to describe the Ellis group of the maximal factor of order $d$. It is also worth noting that throughout this section, we will only assume that $\RP^{[d]}$ is an equivalence relation on distal systems.

For a topological group $T$, let $M_{dis}$ be the associated universal minimal distal system and let $D=\GG(M_{dis})$. The group $D$ is a $\uptau$-closed normal subgroup of $G$ and is generated by $\{g\in G: (p,gp)\in \overline{\P(M)} \text{ for each }p\in M\}$. This group has the property that a minimal topological dynamical system $(X, T)$ is a proximal extension of a distal system if $D\subseteq\GG(X)$. (For more details on this group, see \cite{Auslander_Glasner_distal_order:2002}). By the following results, we can explicitly describe the Ellis group of the maximal factor of order $d$ for a minimal system. 

\begin{lemma}[see {\cite[Chapter 14]{Auslander_minimal_flows_and_extensions:1988}}]\label{lemma: AF_closed_subgroup}
    Let $A,F$ be $\uptau$-closed subgroups of $G$ with $F$ normal in $G$. Then,\begin{enumerate}[label=(\arabic*)]
        \item $AF$ is a $\uptau$-closed subgroup of $G$.
        \item If $g\in G$, $g\circ A = u\circ A$ if and only if $g\in A$.
    \end{enumerate}
\end{lemma}

\begin{theorem}[see {\cite[Chapter 10]{Auslander_minimal_flows_and_extensions:1988}}]\label{thm: A_subset_F_Ydistal_if and only if_X_ext_Y}
    Let $(X,T)$ and $(Y,T)$ be minimal systems, with $Y$ distal. Then $\GG(X)\subseteq \GG(Y)$ if and only if there is an extension $\pi:X\to Y$.
\end{theorem}

\begin{theorem}\label{prop: Ellis_group_nil_d_minimal_system}
        Let $(X, S)$ be a minimal distal topological dynamical system with $A=\GG(X)$ and let $d\geq 1$ be an integer. Then, $\GG(X/\RP^{[d]}(X))=A\Gtau{d+1}$.
\end{theorem}

\begin{proof}
    Since $\Gtau{d+1}$ is trivial on $X/\RP^{[d]}(X)$, we get that $\Gtau{d+1} \subseteq \GG(X/\RP^{[d]}(X))$. Since $X/\RP^{[d]}(X)$ is a factor of $X$, then we have $A\subseteq \GG(X/\RP^{[d]}(X))$.

    It follows from \cref{lemma: AF_closed_subgroup} that $A\Gtau{d+1}$ is a $\uptau$-closed subgroup of $G$. Consider $\Pi(A\Gtau{d+1})=\{p\circ A\Gtau{d+1}:p\in M\}$. Note that $(\Pi(A\Gtau{d+1}), S)$ has a closed proximal relation, so its maximal distal factor $(Y, S)$ satisfies that $\GG(Y)=A\Gtau{d+1}$. Moreover, by \cref{lemma: AF_closed_subgroup}, for each $p\in M$, we have\begin{align*}
        gp \circ A\Gtau{d+1} = upg' \circ A\Gtau{d+1} = up (g' \circ A\Gtau{d+1}) = up\circ A\Gtau{d+1},
    \end{align*}
    where $g'\in \Gtau{d+1}$ satisfies that $gp=upg'$. Then, we get that $\Gtau{d+1}$ is trivial on this system.

    It follows from \cref{thm: A_subset_F_Ydistal_if and only if_X_ext_Y} that $Y$ is an extension of $X/\RP^{[d]}(X)$ and is a factor of $X$. By \cref{prop: Rdtau_subset_RPd}, we have that $X/\RP^{[d]}(X)$ is the maximal distal factor of $X$ with $\Gtau{d+1}$ trivial on it. Hence, we conclude that $Y=X/\RP^{[d]}(X)$.
\end{proof}

Note that by \cref{thm: equi_ext_if and only if_HG_abelian}, using a similar proof to that of \cref{prop: Ellis_group_nil_d_minimal_system}, we can derive \cref{prop: Ellis_group_equi} and \cref{prop: Ellis_group_distal_factor} below. Note that \cref{prop: Ellis_group_equi} appears as an ingredient in the proof of the structure theorem for minimal systems (see \cite[Chapter 14, page 213]{Auslander_minimal_flows_and_extensions:1988}, while \cref{prop: Ellis_group_distal_factor} can be deduced combining \cite[Theorem 4.7, Lemma 5.2 and Theorem 5.3]{Auslander_Glasner_distal_order:2002}.

\begin{proposition}\label{prop: Ellis_group_equi}
    Let $(X, T)$ be a minimal topological dynamical system with $A=\GG(X)$. Then, $\GG(X/\RP(X))=AH(G)$.
\end{proposition}

\begin{proposition}\label{prop: Ellis_group_distal_factor}
    Let $(X, T)$ be a minimal topological dynamical system with $A=\GG(X)$. Then, $\GG(X_{dis})=AD$, where $X_{dis}$ is the maximal distal factor of $X$.
\end{proposition}

Since $X$ has the same maximal factor of order $d$ as its maximal distal factor, we get the following result.

\begin{corollary}\label{prop: Ellis_group_nil_d}
        Let $(X, S)$ be a minimal topological dynamical system with $A=\GG(X)$. Then, $\GG(X_{d})=AD\Gtau{d+1}$, where $X_{d}$ is the maximal factor of order $d$.
\end{corollary}

 Note that \cref{prop: Ellis_group_nil_d} can be extended to any group for which having a $d$-step topologically nilpotent enveloping semigroup is equivalent to being a system of order $d$. In particular, for $d=1$, this characterization is already known in the case of abelian actions on Hausdorff compact spaces (which are not necessarily metric), implying that \cref{prop: Ellis_group_nil_d} also holds in this context. In particular, considering the case where $X=M$, we obtain that $\Gtau{2}\subseteq H(G)$. We do not know whether the reverse inclusion holds. If this inclusion were true, then by \cref{thm: equi_ext_if and only if_HG_abelian} we would obtain a positive answer to \cref{ques:G-nil} in the case $d=1$.

Since $X/\RP^{[\infty]}(X)$ is the inverse limit of $(X/\RP^{[d]}(X))_{d\geq 1}$ for a minimal system $(X,T)$, we can deduce the following results using \cref{prop: Ellis_group_nil_d_minimal_system}.

\begin{proposition}
    Let $(X, S)$ be a minimal topological dynamical system with $\GG(X)=A$. Then, $\GG(X/\RP^{[\infty]}(X))=AD\bigcap_{d\geq 1}G_{d}^{\uptau\mathrm{-top}}$.
\end{proposition}

\begin{corollary}\label{prop: Ellis_groups_nil_infty}
    Let $(X, S)$ be a distal minimal topological dynamical system with $\GG(X)=A$. Then, $\GG(X/\RP^{[\infty]}(X))=A\bigcap_{d\geq 1}G_{d}^{\uptau\mathrm{-top}}$.
\end{corollary}

Note that the group $D$ allows us to characterize the maximal factor of order $d$ (modulo proximal extensions), so it is expected to have a connection to the regionally proximal relation of order $d$. We can establish a characterization of this relation in terms of the group $D$. To prove this characterization, we will not use the strong property that the face transformations acting on the diagonal points are the unique minimal subsystem of the dynamical cube (\cref{thm: cube_is_minimal}, item $(3)$). We believe that our proof strategy could be applied in different contexts, in which the aforementioned condition $(3)$ may not hold. For example, in the case of directional dynamical cubes \cite{Donoso_Sun_cubes_product_ext:2015, Cabezas_Donoso_Maass_directional_cubes:2020}. By using the following result, we can give an algebraic characterization of $\RP^{[d]}$.

\begin{lemma}[{\cite[Lemma 3.3]{Host_Kra_Maass_nilstructure:2010}}]\label{lemma: RPd_iff_xaya_in_cube}
    Let $(X, S)$ be a minimal topological dynamical system and $d\geq 1$ be an integer. Then $(x,y)\in \RP^{[d]}(X)$ if and only if there exists $a_{*}\in X^{2^{d}-1}$ such that $(x,a_{*},y,a_{*})\in \bQ^{[d+1]}(X)$.
\end{lemma}

\begin{theorem}\label{lemma: Rdis_is_RPd}    
    Let $(X, S)$ be a minimal topological dynamical system and let $d\geq 1$ be an integer. Then, the relation  $R_{d}^{\mathrm{dis}}(X)$ defined as
    \begin{align*}
        R_{d}^{\mathrm{dis}}(X):=\{(x,vghx): v\in J(M), g\in \Gtau{d+1},h\in D, x\in X\},
    \end{align*}
equals $\RP^{[d]}(X)$. Moreover, $\RP^{[d]}(X)$ is an equivalence relation.
\end{theorem}

\begin{proof}
    Let $x\in X$ and let $g\in G$ such that $(p,gp)\in\overline{\P(M)}$ for every $p\in M$. In particular, we have $(x,gx)\in \RP^{[d]}(X)$. It follows from \cref{lemma: RPd_iff_xaya_in_cube} that there exists $\bold{a}_{*}\in X^{2^{d}-1}$ such that $(x,\bold{a}_{*},gx,\bold{a}_{*})\in\bQ^{[d+1]}(X)$. Therefore, by \cref{thm: dyn_cube_equal_HK_cube}, we have $(ux,\bold{u}\bold{a}_{*},ugx,\bold{u}\bold{a}_{*})\in\cltau{\HK{d+1}}x_{0}^{[d+1]}$, where $\bold{u} = (u,\dots,u)\in G^{2^{d}-1}$. Since $D$ is generated by $\{g\in G: (p,gp)\in\overline{\P(M)}\text{ for every }p\in M\}$, we get that $(ux,\bold{u}\bold{a}_{*},hx,\bold{u}\bold{a}_{*})\in\cltau{\HK{d+1}}x_{0}^{[d+1]}$ for each $h\in D$, so $(ux,hx)\in \RP^{[d]}(X)$ for every $h\in D$.
    
     Let $g\in \Gtau{d+1}$ and  let $(x,y)\in \RP^{[d]}(X)$. By \cref{lemma: RPd_iff_xaya_in_cube}, there exists $\bold{a}_{*}\in X^{2^{d}-1}$ such that $(x,\bold{a}_{*},y,\bold{a}_{*})\in \bQ^{[d+1]}(X)$. Since $g\in \Gtau{d+1}$, by \cref{lemma: HK_invaraint_g_alpha} and \cref{thm: dyn_cube_equal_HK_cube}, it follows that $(ux,\bold{g}\bold{a}_{*},gy,\bold{g}\bold{a}_{*})\in \bQ^{[d+1]}(X)$, where $\bold{g} =(g,\dots,g)\in G^{2^{d}-1}$. Therefore, \cref{lemma: RPd_iff_xaya_in_cube} implies that $(ux,gy)\in \RP^{[d]}(X)$. Thus, by a preceding argument, we deduce that $(ux,ghx)\in \RP^{[d]}(X)$ for each $x\in X$, $h\in D$ and $g\in \Gtau{d+1}$.

    Now, let $(x,y)\in \RP^{[d]}(X)$ and $v,w\in J(M)$. By \cref{lemma: RPd_iff_xaya_in_cube}, there exists $\bold{a}_{*}\in X^{2^{d}-1}$ such that $(x,\bold{a}_{*},y,\bold{a}_{*})\in \bQ^{[d+1]}(X)$. It follows from \cref{lemma: HK_Ellis_contenido_cubo} that $(vx,\bold{w}\bold{a}_{*},wy,\bold{w}\bold{a}_{*})\in \bQ^{[d+1]}(X)$, where $\bold{w} =(w,\dots,w)\in M^{2^{d}-1}$. Therefore, \cref{lemma: RPd_iff_xaya_in_cube} implies that $(vx,wy)\in \RP^{[d]}(X)$. Thus, we obtain that $(x,vghx)\in \RP^{[d]}(X)$ for each $x\in X$, $v\in J(M)$, $h\in D$ and $g\in \Gtau{d+1}$, and hence we conclude that $R_{d}^{dis}(X)\subseteq \RP^{[d]}(X)$.

    Let $(x,y)\in \RP^{[d]}(X)$ and let $\pi: X\to X_{d}$ the factor map to the maximal factor of order $d$. Then, we have\begin{align*}
        \pi\times\pi(x,y)\in \RP^{[d]}(X_{d})=R_{d}^{\uptau}(X_{d})=\Delta.
    \end{align*}

    Let $p,q\in M$ be such that $y=px_{0}$ and $x=qx_{0}$. Since $\pi(y)=\pi(x)$, we obtain that $up^{-1}q\pi(x_{0})=\pi(x_{0})$. It follows from \cref{prop: Ellis_group_nil_d_minimal_system} that $up^{-1}q \in \GG(X_{d})= \Gtau{d+1}DA$. In particular, we have $up^{-1}qx_{0} \in \Gtau{d+1}Dx_{0}$. Thus, there exists $g\in D$ and $h\in \Gtau{d+1}$ such that $upx_{0}=hgqx_{0}$. Let $v\in J(M)$ be such that $vy=y$.  Hence, we have\begin{align*}
        y = vupx_{0} = vhgqx_{0} = vhgx,
    \end{align*}
    
    completing the proof.

To conclude that $\RP^{[d]}(X)$ is an equivalence relation, it suffices to show that  $R_{d}^{dis}$ is an equivalence relation. The proof of this is identical to that in \cref{thm: Rdtau_closed}, but we include it here for completeness. Let $(x,vgx),(x,whx)\in R_{d}^{dis}(X)$, where $v,w\in J(M)$ and $g,h\in G_{d+1}^{\uptau\mathrm{-top}}D$. Note that $hg^{-1} \in \Gtau{d+1}D$. Then, we have\begin{align*}
        (vgx,whx)=(vgx,w(hg^{-1})vgx) \in R_{d}^{dis}(X).
    \end{align*}

Therefore, $R_{d}^{dis}(X)$ is an equivalence relation.
\end{proof}

Let us note that, for the characterization of the above theorem, we only use the regionally proximal relations in distal systems and then extend them to an arbitrary minimal system using the group $D$. Observe that \cref{lemma: Rdis_is_RPd} can be extended to any group such that in which having a $d$-step topologically nilpotent enveloping semigroup is equivalent to being a system of order $d$.

Using the algebraic characterization of the regionally proximal relation, we can give a new and short proof of the lifting property of $\RP^{[d]}$.

\begin{corollary}\label{cor: lifting_RPd}
    Let $\pi\colon (X,S)\to (Y, S)$ be an extension of minimal systems and $d\geq1$ be an integer. Then $\pi\times\pi(\RP^{[d]}(X))=\RP^{[d]}(Y)$.
\end{corollary}

\begin{proof}
    It is clear that $\pi\times\pi(\RP^{[d]})(X) \subseteq \RP^{[d]}(Y)$, so we only need to prove the reverse inclusion. Let $(y,y')\in \RP^{[d]}(Y)$. By \cref{lemma: Rdis_is_RPd}, there exists $v\in J(M)$, $d\in D$ and $g\in \Gtau{d+1}$ such that $y'=vdgy$. Consider $x\in X$ such that $y=\pi(x)$. Then, by \cref{lemma: Rdis_is_RPd} we have $(x,vdgx)\in \RP^{[d]}(X)$, and $\pi(vdgx)=y'$. This finishes the proof.
\end{proof}

\subsection{Topological characteristic factor along cubes} \label{sec:applications_charfac_cube}

The counterpart of characteristic factors in topological dynamics was first studied by Glasner in \cite{Glasner_top_erg_decomposition:1994}. In the recent work by Cai and Shao in \cite{Cai_Shao_Topological_characteristic_cubes:2019}, they introduced the concept of the topological cubic characteristic factor of order $d$.

\begin{definition}
    Let $(X, T)$ be a topological dynamical system and let $\pi:(X, T)\to (Y, T)$ be an extension. A subset $L$ of $X$ is called {\em $\pi$-saturated} if $L=\pi^{-1}(\pi(L))$.
\end{definition}

The definition of the topological cubic characteristic factor of order $d$ is as follows:

\begin{definition}
    Let $(X, T)$ be a topological dynamical system and let $\pi:(X, T)\to (Y, T)$ be an extension. The system $(Y, T)$ is said to be a {\em topological cubic characteristic factor of order $d$} if there exists a dense $G_{\delta}$ set $X_{0}$ of $X$ such that for each $x\in X_{0}$ the set $\rho_{*}(\overline{\mathcal{F}^{[d]}(T)x^{[d]}})$ is $\pi^{2^{d}-1}$-saturated, where $\rho_{*}: X^{[d]}\to X^{2^{d}-1}$ is the projection on the latter $2^{d}-1$ coordinates.
\end{definition}

Cai and Shao \cite{Cai_Shao_Topological_characteristic_cubes:2019} established that, up to proximal extensions, the maximal factor of order $d-1$ is the topological cubic characteristic factor of order $d$. In this subsection, we will extend these results to almost one-to-one extensions, as well as to systems of order $\infty$. Before, we need the following result.

\begin{lemma}\label{lemma: Ellis_distal_group_cube_universal}
    Let $(X, T)$ be a minimal topological dynamical system and let $d\geq 1$ be an integer. Then,  
    $$D^{[d]} \subseteq \cltau{\HK{d}}.$$ 
In particular, $D^{[d]}\bx \subseteq \bQ^{[d]}(X)$ for all $\bx\in \bQ^{[d]}(X)$.
\end{lemma}

\begin{proof}
    Since $\mathcal{HK}^{[1]}(G)=G\times G$, we immediately get the conclusion for $d=1$.
    Now suppose that $d\geq 2$. Let $g\in D$. Since $D$ is a normal subgroup of $G$, similar to the proof of \cref{prop: Ellis_group_nil_d_minimal_system}, we have that $D$ is trivial on $M_{dis}$, where $M_{dis}$ is the maximal distal factor of $M$. So, $D$ is also trivial on $M/\RP^{[d-1]}(M)$, meaning that $g=u$ on $M/\RP^{[d-1]}(M)$. Therefore, we get that $(u,g)\in \RP^{[d-1]}(M)$. Thus, we have\begin{align*}
        (u,\dots,u,g) \in u^{[d]}\bQ^{[d]}(M)=\cltau{\HK{d}}.
    \end{align*}

   Using Euclidean permutations, and the fact that $\cltau{\HK{d}}$ is invariant under them, we get $D^{[d]} \subseteq \cltau{\HK{d}}$.  By \cref{thm: dyn_cube_equal_HK_cube}, we conclude that $D^{[d]}\bx\subseteq \bQ^{[d]}(X)$ for all $\bx\in \bQ^{[d]}(X)$.
\end{proof}

The next proposition improves \cite[Proposition 3.3]{Cai_Shao_Topological_characteristic_cubes:2019} from RIC to open extensions.  

\begin{proposition}\label{prop: cube_saturated_open_ext}
     Let $\pi: (X, S)\to (Y, S)$ be an extension of minimal systems and $d\geq 2$ be an integer. If $\pi$ is open and $X_{d-1}$ is a factor of $Y$, where $X_{d-1}$ is the maximal factor of order $d-1$ of $X$, then $\bQ^{[d]}(X)$ is $\pi^{[d]}$-saturated.
\end{proposition}

\begin{proof}
        First, we will show that $\overline{J(M)}\subseteq J(M)D$. Let $\psi: M\to M_{dis}$ be the quotient map, where $M_{dis}$ is the maximal distal factor of $M$. Let $p\in \overline{J(M)}$, then there exists $v\in J(M)$ and $g\in G$ such that $p=vg$. Since $\psi$ is continuous, we have $\psi(\overline{J(M)}) \subseteq \overline{\psi(J(M))} = \{\psi(u)\}$. Therefore, we get that\begin{align*}
        \psi(u)=\psi(p)=vg\psi(u) =g\psi(u).
        \end{align*}

    Thus, we have $g\in D$, and we conclude that $\overline{J(M)}\subseteq J(M)D$.

     It follows from \cref{thm: dyn_cube_equal_HK_cube} that\begin{align*}
        (\pi^{[d]})^{-1}(\pi^{[d]} (\bQ^{[d]}(X)) ) &= (\pi^{[d]})^{-1}(\pi^{[d]} (\cltau{\HK{d}}x_{0}^{[d]})\\
        &= (\pi^{[d]})^{-1}(\cltau{\HK{d}}\pi^{[d]} (x_{0}^{[d]})).     \end{align*} 
        
    Let $\bx\in (\pi^{[d]})^{-1}(\cltau{\HK{d}}\pi^{[d]} (x_{0}^{[d]}))$. Consider $\bg\in \cltau{\HK{d}}$ such that $\pi^{[d]}(\bx)=\pi^{[d]}(\bg x_{0}^{[d]})$. Since $\pi$ is open, by \cref{thm: extensions_ellis_group}, we have \begin{align*}
        (\pi^{[d]})^{-1}(\bg\pi^{[d]}(x_{0}^{[d]}))= \bg \circ (\pi^{[d]})^{-1}(\pi^{[d]}(x_{0}^{[d]})).
    \end{align*}
Let $\pi_{d-1}\colon X\to X_{d-1}$ and $\phi\colon Y \to X_{d-1}$ be the natural factor maps. Note that\begin{align*}
        \bg \circ (\pi^{[d]})^{-1}(\pi^{[d]}(x_{0}^{[d]})) \subseteq \bg \circ (\pi^{[d]})^{-1}((\phi^{[d]})^{-1}( \phi^{[d]}(\pi^{[d]}(x_{0}^{[d]})))) = \bg \circ (\pi_{d-1}^{[d]})^{-1}(\pi_{d-1}^{[d]}(x_{0}^{[d]})).
    \end{align*}

    By \cref{lemma: Rdis_is_RPd}, we have $(\pi_{d-1}^{[d]})^{-1}(\pi_{d-1}^{[d]}(x_{0}^{[d]})) = J(M)^{[d]}(\Gtau{d})^{[d]}D^{[d]}x_{0}^{[d]}$.
    
    Let $\bp\in u^{{[d]}} \circ J(M)^{[d]}(\Gtau{d})^{[d]}D^{[d]}$. Then, there exist $(\bt_{\lambda})_{\lambda\in\Lambda}$, $(\bv_{\lambda})_{\lambda\in\Lambda}$, $(\bold{d}_{\lambda})_{\lambda\in\Lambda}$ and $(\bg_{\lambda})_{\lambda\in\Lambda}$ nets in $\mathcal{HK}^{[d]}(S)$, $J(M)^{[d]}$, $D^{[d]}$ and $(\Gtau{d})^{[d]}$ respectively, such that $\bt_{\lambda} \to u^{[d]}$ and $\bt_{\lambda}\bv_{\lambda}\bold{d}_{\lambda}\bg_{\lambda} \to \bp$. Consider $\bv\in J(M)^{[d]}$ and $\bold{d}\in D^{[d]}$ such that $\bv_{\lambda} \to \bv \bold{d}$. By \cref{lemma: conv_prod_tau_top}, we get that\begin{align*}
        \bt_{\lambda}\bold{dd}_{\lambda}\bg_{\lambda} \taulim u^{[d]}\bp.
    \end{align*}

    Let $\bold{d}'\in D^{[d]}$ and $\bg\in (\Gtau{d})^{[d]}$ be such that $\bold{d}_{\lambda} \taulim \bold{d}'$ and $\bg_{\lambda} \taulim \bg$. Since $\bt_{\lambda}\to u^{[d]}$, we have that\begin{align*}
        \bt_{\lambda}\bold{dd}_{\lambda}\bg_{\lambda}H(G(\mathcal{HK}^{[d]}(S))) \taulim \bold{dd}'\bg H(G(\mathcal{HK}^{[d]}(S))) =  u^{[d]}\bp H(G(\mathcal{HK}^{[d]}(S))).
    \end{align*}

    Therefore, there is a $\bh\in H(G(\mathcal{HK}^{[d]}(S)))$ such that $u^{[d]}\bp = \bold{dd'}\bg\bh$. Thus,\begin{align*}
        u^{{[d]}} \circ J(M)^{[d]}(\Gtau{d})^{[d]}D^{[d]} \subseteq J(M(\mathcal{HK}^{[d]}(S))D^{[d]}(\Gtau{d})^{[d]}H(G(\mathcal{HK}^{[d]}(S))). 
    \end{align*}

    By \cref{lemma: HK_invaraint_g_alpha} and  \cref{lemma: Ellis_distal_group_cube_universal}, we get that $D^{[d]}(\Gtau{d})^{[d]}x_{0}^{[d]} \subseteq \cltau{\HK{d}}$. Therefore, from \cref{thm: dyn_cube_equal_HK_cube}, it follows that
    \begin{align*}
        J(M(\mathcal{HK}^{[d]}(S))D^{[d]}(\Gtau{d})^{[d]}H(G(\mathcal{HK}^{[d]}(S)))x_{0}^{[d]} \subseteq J(M(\mathcal{HK}^{[d]}(S))\cltau{\HK{d}}x_{0}^{[d]} = \bQ^{[d]}(X).
    \end{align*}
    
    Thus, \begin{align*}
        \bg \circ (\pi^{[d]})^{-1}(\pi^{[d]}(x_{0}^{[d]})) = \bg\circ (u \circ (\pi^{[d]})^{-1}(\pi^{[d]}(x_{0}^{[d]}))) \subseteq \bg\circ \cltau{\HK{d}}x_{0}^{[d]} \subseteq \bg \circ \bQ^{[d]}(X).
    \end{align*}

    Since $\bQ^{[d]}(X)$ is an $\mathcal{HK}^{[d]}(S)$-invariant closed set, we have that $\bg \circ \bQ^{[d]}(X) = \bQ^{[d]}(X)$. Hence, we conclude that $\bQ^{[d]}(X)$ is $\pi^{[d]}$-saturated.
\end{proof}

Note that for all $d \geq 1$, we have $\bigcap_{d\geq 1} \Gtau{d} \subseteq \Gtau{j}$ for $j\geq 1$. Using the same proof as in the \cref{prop: cube_saturated_open_ext} and using \cref{prop: Ellis_groups_nil_infty}, we can conclude the following result.

\begin{proposition}\label{prop: cube_saturated_open_ext_nil_infty}
     Let $\pi: (X, S)\to (Y, S)$ be an extension of minimal systems. If $\pi$ is open and $X_{\infty}$ is a factor of\, $Y$, where $X_{\infty}$ is the maximal factor of order $\infty$ of $X$, then $\bQ^{[d]}(X)$ is $\pi^{[d]}$-saturated for each $d\geq 1$.
\end{proposition}

Every extension of minimal systems can be lifted to an open extension.

\begin{theorem}[see {\cite[Chapter VI]{deVries_elements_topological_dynamics:1993}}]\label{thm: AG_diagram}
    Given $\pi: (X, T) \to (Y, T)$ an extension of minimal systems, there exists a commutative diagram of extensions (called AG-diagram)
    \[\begin{tikzcd}
	X && {X'} \\
	\\
	{Y} && {Y'}
	\arrow["\pi"', from=1-1, to=3-1]
	\arrow["{\theta'}"', from=1-3, to=1-1]
	\arrow["{\pi'}", from=1-3, to=3-3]
	\arrow["\theta", from=3-3, to=3-1]
\end{tikzcd}\]

    such that\begin{enumerate}[label=(\arabic*)]
        \item $\theta$ and $\theta'$ are almost one-to-one extensions,
        \item $\pi'$ is an open extension,
        \item $X'$ is the unique minimal set in $R_{\pi\theta} = \{(x,y)\in X\times Y': \pi(x)=\theta(y)\}$ and $\theta'$ and $\pi'$ are the restriction to $X'$ of the projections of $X\times Y'$ onto $X$ and $Y'$ respectively.
    \end{enumerate}
\end{theorem}

Therefore, using the AG-diagram, we can conclude that modulo almost one-to-one extensions $\bQ^{[d]}(X)$ is $\pi^{[d]}$-saturated.

\begin{theorem}\label{thm: saturation_cube_minimal_system}
        Let $(X, S)$ be a minimal topological system and $d\geq 2$ be an integer. Let $\pi:(X, S)\to (X_{d-1}, S)$ be the extension to the maximal factor of order $d-1$. Then, there is a commutative diagram of extensions of minimal systems

        \[\begin{tikzcd}
	X && {X'} \\
	\\
	{X_{d-1}} && {X_{d-1}'}
	\arrow["\pi"', from=1-1, to=3-1]
	\arrow["{\theta'}"', from=1-3, to=1-1]
	\arrow["{\pi'}", from=1-3, to=3-3]
	\arrow["\theta", from=3-3, to=3-1]
    \end{tikzcd}\]

    such that $\bQ^{[d]}(X') = ({\pi'}^{[d]})^{-1}(\bQ^{[d]}(X_{d-1}'))$, where $\theta$ and $\theta'$ are almost one-to-one extensions.
\end{theorem}

\cref{thm: saturation_cube_minimal_system} is the best result we can expect, meaning that the almost one-to-one modifications are necessary (see {\cite[Example 3.7]{Cai_Shao_Topological_characteristic_cubes:2019}}). Clearly, we have the same theorem for the maximal factor of order $\infty$.

\begin{theorem}\label{thm: saturation_cube_minimal_system_infty}
        Let $(X, S)$ be a minimal topological system. Let $\pi:(X, S)\to (X_{\infty}, S)$ be the extension to the maximal factor of order $\infty$. Then, there is a commutative diagram of extensions of minimal systems

        \[\begin{tikzcd}
	X && {X'} \\
	\\
	{X_{\infty}} && {X_{\infty}'}
	\arrow["\pi"', from=1-1, to=3-1]
	\arrow["{\theta'}"', from=1-3, to=1-1]
	\arrow["{\pi'}", from=1-3, to=3-3]
	\arrow["\theta", from=3-3, to=3-1]
    \end{tikzcd}\]

    such that $\bQ^{[d]}(X') = ({\pi'}^{[d]})^{-1}(\bQ^{[d]}(X_{\infty}'))$ for each $d\geq 1$, where $\theta$ and $\theta'$ are almost one-to-one extensions.
\end{theorem}

By \cref{thm: cube_is_minimal}, for a minimal system there exists a dense $G_{\delta}$ subset $X_{0}\subseteq X$ such that $\bQ^{[d]}(X)\cap (\{x\} \times X^{2^{d}-1})$ equals $\overline{\mathcal{F}^{[d]}(T)x^{[d]}}$ for each $x\in X_{0}$. From this, we can conclude the following theorems.

\begin{proposition}\label{prop: top_car_factor_d}
    Let $(X, S)$ be a minimal topological system and $d\geq 2$ be an integer. Let $\pi:(X, S)\to (X_{d-1}, S)$ be the extension to the maximal factor of order $d-1$. Then, there is a commutative diagram
    \[\begin{tikzcd}
	X && {X'} \\
	\\
	{X_{d-1}} && {X_{d-1}'}
	\arrow["\pi"', from=1-1, to=3-1]
	\arrow["{\theta'}"', from=1-3, to=1-1]
	\arrow["{\pi'}", from=1-3, to=3-3]
	\arrow["\theta", from=3-3, to=3-1]
\end{tikzcd}\]
such that $(X_{d-1}',S)$ is the topological cubic characteristic factor of order $d$ of $(X',S)$, where $\theta,\theta'$ are almost one-to-one extensions.
\end{proposition}

\begin{proposition}\label{prop: top_car_factor_infty}
    Let $(X, S)$ be a minimal topological system. Let $\pi:(X, S)\to (X_{\infty}, S)$ be the extension to the maximal factor of order $\infty$. Then, there is a commutative diagram
    \[\begin{tikzcd}
	X && {X'} \\
	\\
	{X_{\infty}} && {X_{\infty}'}
	\arrow["\pi"', from=1-1, to=3-1]
	\arrow["{\theta'}"', from=1-3, to=1-1]
	\arrow["{\pi'}", from=1-3, to=3-3]
	\arrow["\theta", from=3-3, to=3-1]
\end{tikzcd}\]
such that $(X_{\infty}',S)$ is the topological cubic characteristic factor of order $d$ of $(X',S)$ for each $d\geq 1$, where $\theta,\theta'$ are almost one-to-one extensions.
\end{proposition}

It is worth noting that the diagrams in \cref{thm: saturation_cube_minimal_system}, \cref{thm: saturation_cube_minimal_system_infty}, \cref{prop: top_car_factor_d} and \cref{prop: top_car_factor_infty} also satisfy conditions $(2)$ and $(3)$ of \cref{thm: AG_diagram}, although our primary focus was on condition $(1)$.

\subsection{Structure theorem of dynamical cubes} \label{Sec:structure_theorem} In this subsection, we will show that for abelian actions, a minimal system $(X,T)$ and its associated system of cubes $(\bQ^{[d]}(X), \mathcal{HK}(T))$ share a similar structural description, which we will clarify further below.

A minimal system $(X, T)$ is called a {\em strictly PI system} if there exists an ordinal $\eta$ (which is countable when $X$ is metrizable) and a family of systems $\{(W_{\iota}, T)\}_{\iota \leq \eta}$ satisfying the following conditions: (i) $W_{0}$ is the trivial system, (ii) for every $\iota < \eta$ there is an extension $\phi_{\iota}: W_{\iota + 1} \to W_{\iota}$, that is either proximal or equicontinuous, (iii) for a limit ordinal $\nu \leq \eta$, $W_{\nu}$ is the inverse limit of the systems $\{W_\iota\}_{\iota < \nu}$, and (iv) $W_{\eta} = X$. A minimal system $(X, T)$ is called a {\em PI-system} if there exists a strictly PI system $X_{\infty}$ and a proximal extension $\theta: X_{\infty} \to X$.

If, in the definition of PI systems,  the extensions  $\phi_{\iota}: W_{\iota + 1} \to W_{\iota}$  are either almost one-to-one or equicontinuous, we obtain the notion of {\em AI systems.} Note that any AI system is a PI system. If we further ask that all the $\phi_{\iota}$ are equicontinuous (meaning that no nontrivial proximal extensions are allowed),  we obtain the notion of {\em I systems}. In this terminology, the structure theorem for distal systems can be stated as follows:

\begin{theorem}[{\cite[Theorems 2.3 and 2.4]{Furstenberg_structure_distal_flows:1963}}]
    A metric minimal system is distal if and only if it is an $I$-system.
\end{theorem}

A minimal system is called {\em point distal} if there exists a point that is proximal only to itself. The Veech's structure theorem for point distal systems then states:

\begin{theorem}[{\cite[Theorem 7.2]{Veech_point-distal_flows:1970}}]
    A metric minimal system is point distal if and only if it is an $AI$-system.
\end{theorem}

Finally, we have the structure theorem for minimal systems:

\begin{theorem}[{\cite[Proposition 7.3]{Ellis_Glasner_Shapiro_PI-flows:1975},\cite[Theorem 1.3]{McMahon_Weak_mixing_structure_thm:1976},\cite[Theorem 2.1.3]{Veech_topological_dynamics:1977}}]
    Let $(X, T)$ be a minimal topological dynamical system. Then, there exists an ordinal $\eta$ and a canonically defined commutative diagram (the canonical PI-Tower)
    \[\begin{tikzcd}[cramped,column sep=small]
	X && {} && {X_{1}} & \cdots & {X_{\nu}} &&&& {X_{\nu+1}} & \cdots & {X_{\eta}} \\
	\\
	{\{\cdot\}} && {Z_{1}} && {Y_{1}} & \cdots & {Y_{\nu}} && {Z_{\nu+1}} && {Y_{\nu+1}} & \cdots & {Y_{\eta}}
	\arrow["{\theta^{*}_{1}}", from=1-5, to=1-1]
	\arrow["\pi"', from=1-1, to=3-1]
	\arrow["{\sigma_{1}}", from=1-1, to=3-3]
	\arrow["{\pi_{1}}"', from=1-5, to=3-5]
	\arrow["{\theta^{*}_{\nu+1}}", from=1-11, to=1-7]
	\arrow["{\pi_{\nu}}"', from=1-7, to=3-7]
	\arrow["{\sigma_{\nu+1}}", from=1-7, to=3-9]
	\arrow["{\pi_{\nu+1}}"', from=1-11, to=3-11]
	\arrow["{\pi_{\eta}}"', from=1-13, to=3-13]
	\arrow["{\rho_{1}}", from=3-3, to=3-1]
	\arrow["{\theta_{1}}", from=3-5, to=3-3]
	\arrow["{\rho_{\nu+1}}", from=3-9, to=3-7]
	\arrow["{\theta_{\nu+1}}", from=3-11, to=3-9]
\end{tikzcd}\]

    where for each $\nu\leq \eta$, $\pi_{\nu}$ is RIC, $\rho_{\nu}$ is equicontinuous, $\theta_{\nu},\theta^{*}_{\nu}$ are proximal and $\pi_{\eta}$ is RIC and weakly mixing. For a limit ordinal $\nu$, $X_{\nu},Y_{\nu},\pi_{\nu}$, etc. are the inverse limits (or joins) of $X_{\iota}, Y_{\iota},\pi_{\iota}$, etc for $\iota<\nu$.
\end{theorem}

Then, it naturally induces the following commutative diagram associated with $\bQ^{[d]}(X)$:

\[\adjustbox{width={\textwidth}}{\begin{tikzcd}[cramped,column sep=small]
	\bQ^{[d]}(X) && {} && {\bQ^{[d]}(X_{1})} & \cdots & {\bQ^{[d]}(X_{\nu})} &&&& {\bQ^{[d]}(X_{\nu+1})} & \cdots & {\bQ^{[d]}(X_{\eta})} \\
	\\
	{\{\cdot\}} && {\bQ^{[d]}(Z_{1})} && {\bQ^{[d]}(Y_{1})} & \cdots & {\bQ^{[d]}(Y_{\nu})} && {\bQ^{[d]}(Z_{\nu+1})} && {\bQ^{[d]}(Y_{\nu+1})} & \cdots & {\bQ^{[d]}(Y_{\eta})}
	\arrow["{{\theta^{*}_{1}}^{[d]}}", from=1-5, to=1-1]
	\arrow["\pi^{[d]}"', from=1-1, to=3-1]
	\arrow["{\sigma_{1}^{[d]}}", from=1-1, to=3-3]
	\arrow["{\pi_{1}^{[d]}}"', from=1-5, to=3-5]
	\arrow["{{\theta^{*}_{\nu+1}}^{[d]}}", from=1-11, to=1-7]
	\arrow["{\pi_{\nu}^{[d]}}"', from=1-7, to=3-7]
	\arrow["{\sigma_{\nu+1}^{[d]}}", from=1-7, to=3-9]
	\arrow["{\pi_{\nu+1}^{[d]}}"', from=1-11, to=3-11]
	\arrow["{\pi_{\eta}^{[d]}}"', from=1-13, to=3-13]
	\arrow["{\rho_{1}^{[d]}}", from=3-3, to=3-1]
	\arrow["{\theta_{1}^{[d]}}", from=3-5, to=3-3]
	\arrow["{\rho_{\nu+1}^{[d]}}", from=3-9, to=3-7]
	\arrow["{\theta_{\nu+1}^{[d]}}", from=3-11, to=3-9]
\end{tikzcd}}\]

In this section, we will show the following theorem.

\begin{theorem}\label{thm: Structure_thm_cube}
    Let $(X, T)$ be a minimal topological dynamical system and let $d\geq 1$ be an integer. Then, $(\bQ^{[d]}(X),\mathcal{HK}^{[d]}(T))$ has the same structure theorem as $(X, T)$. Precisely, for each $\nu\leq \eta$, $\pi_{\nu}^{[d]}$ is RIC, $\rho_{\nu}^{[d]}$ is equicontinuous, $\theta_{\nu}^{[d]},{\theta^{*}_{\nu}}^{[d]}$ are proximal and $\pi_{\eta}^{[d]}$ is RIC and weakly mixing.
\end{theorem}

\cref{thm: Structure_thm_cube} is an immediate consequence of the following proposition.

\begin{proposition}\label{prop: ext_P_then_cubes_ext_P}
    If $\pi:X\to Y$ is an extension between minimal systems with property $P$, then the extension $\pi^{[d]}:\bQ^{[d]}(X)\to \bQ^{[d]}(Y)$ also has property $P$. Here, $P$ can refer to being distal, proximal, equicontinuous, RIC or RIC weakly mixing.
\end{proposition}

An interesting question is whether \cref{prop: ext_P_then_cubes_ext_P} remains valid when $P$ is being weakly mixing. In our proof, the assumption of RIC weakly mixing is required to deduce that the extension between the dynamical cubes is weakly mixing, in fact, we need the factor to be totally weakly mixing. Currently, we do not know how to prove \cref{prop: ext_P_then_cubes_ext_P} only assuming that $P$ is weakly mixing. Weakly but not totally weakly mixing extensions are natural candidates for counterexamples to this property (for such examples, see, for instance, {\cite[Theorem 6.12]{Cao_Shao_almost_proximal:2023}}). 

We prove \cref{prop: ext_P_then_cubes_ext_P} case by case, and for that, we analyze the Ellis group of a system of dynamical cubes. The following lemma describes this object, and its proof follows directly from the definitions.

\begin{lemma}\label{prop: Ellis_group_cube}
    Let $(X, T)$ be a minimal topological dynamical system with $\GG(X)=A$. Then, $\GG(\bQ^{[d]}(X))=\cltau{\HK{d}}\cap A^{[d]}$, $d\geq 1$.
\end{lemma}

By \cref{prop: Ellis_group_cube} and \cref{thm: extensions_ellis_group}, we get immediately the following result.

\begin{proposition}\label{prop: pi_prox_pid_prox_in_cubes}
    Let $\pi: (X, T) \to (Y, T)$ be a proximal extension of minimal systems. Then, $\pi^{[d]}: (\bQ^{[d]}(X),\mathcal{HK}^{[d]}(T)) \to (\bQ^{[d]}(Y),\mathcal{HK}^{[d]}(T))$ is a proximal extension.
\end{proposition}

First, we will show that if $\pi: (X, T) \to (Y, T)$ is a RIC extension (resp. RIC weakly mixing) between minimal systems, then so is $\pi^{[d]}: \bQ^{[d]}(X)\to \bQ^{[d]}(Y)$. We need the following results.

\begin{lemma}[see {\cite[Chapter 14]{Auslander_minimal_flows_and_extensions:1988}}]\label{lemma: RIC_iff_open_preim_is_ucircFx0}
    Let $\pi: (X, T) \to (Y, T)$ be an extension of minimal systems. Then, $\pi$ is RIC if and only if it is open and $\pi^{-1}(y_{0})=u\circ \GG(Y)x_{0}$.
\end{lemma}

\begin{theorem}[{\cite[Theorem 4.3]{Shao_Ye_regionally_prox_orderd:2012}}]
    Let $\pi: (X, T) \to (Y, T)$ be a RIC weakly mixing extension of minimal systems. Then, for all $n\geq 1$ and $y\in Y$, there exists a transitive point $(x_{1},\dots,x_{n})\in R_{\pi}^{n}$ with $x_{1},\dots,x_{n}\in \pi^{-1}(y)$.
\end{theorem}
    
\begin{lemma}[{\cite[Lemma 4.4]{Shao_Ye_regionally_prox_orderd:2012}}]\label{lemma: preimage_d_of_y_in_cube}
    Let $\pi: (X, T) \to (Y, T)$ be a RIC weakly mixing extension of minimal systems. Then, for each $y\in Y$ and $d\geq 1$, we have\begin{align*}
        (\pi^{-1}(y))^{[d]} = (\pi^{-1}(y))^{2^{d}} \subseteq \bQ^{[d]}(X).
    \end{align*}
\end{lemma}

\begin{proposition}
    Let $\pi: (X, T) \to (Y, T)$ be a RIC extension of minimal systems. Then, $\pi^{[d]}: (\bQ^{[d]}(X),\mathcal{HK}^{[d]}(T)) \to (\bQ^{[d]}(Y),\mathcal{HK}^{[d]}(T))$ is a RIC extension.
\end{proposition}

\begin{proof}
    Since $\pi$ is RIC, we have \begin{align*}
        (\pi^{[d]})^{-1}(y_{0}^{[d]}) = (u \circ Fx_{0})^{[d]} \cap \bQ^{[d]}(X) = u^{[d]} \circ F^{[d]}x_{0}^{[d]} \cap \bQ^{[d]}(X),
    \end{align*}

    where $y_{0}=\pi(x_{0})$ and $F=\GG(Y)$. Since $\bQ^{[d]}(X)$ is an $\mathcal{HK}^{[d]}(T)$-invariant closed set, we have $u^{[d]} \circ \bQ^{[d]}(X) = \bQ^{[d]}(X)$. It follows from \cref{thm: dyn_cube_equal_HK_cube} that $u^{[d]} \circ \cltau{\HK{d}}$ is contained in $u^{[d]}\circ \bQ^{[d]}(X)$. Then, \begin{align*}
        u^{[d]} \circ (F^{[d]}\cap \cltau{\mathcal{HK}^{[d]}(G)})x_{0}^{[d]} \subseteq u^{[d]} \circ F^{[d]}x_{0}^{[d]} \cap \bQ^{[d]}(X)
    \end{align*}

    Let $\bx\in u^{[d]} \circ F^{[d]}x_{0}^{[d]} \cap \bQ^{[d]}(X)$. Consider $\bp\in u^{[d]} \circ F^{[d]}$, $\bv\in J(M(\mathcal{HK}^{[d]}(T)))$ and $\bg \in\cltau{\HK{d}}$ such that $x=\bp x_{0}^{[d]}=\bv\bg x_{0}^{[d]}$. Then, we get that $\bg(\epsilon)\bp(\epsilon)^{-1} \in A \subseteq F$ for each $\epsilon\subseteq [d]$.

    Since $\bp\in u^{[d]} \circ F^{[d]}$, there are $(\bt_{\lambda})_{\lambda\in\Lambda}$ and $(\bold{f}_{\lambda})_{\lambda\in\Lambda}$ nets in $\mathcal{HK}^{[d]}(T)$ and $F^{[d]}$ respectively such that $\bt_{\lambda} \to u^{[d]}$ and $\bt_{\lambda}\bold{f}_{\lambda} \to \bp$. It follows from \cref{lemma: conv_prod_tau_top} that\begin{align*}
        \bold{f}_{\lambda}(\epsilon) \taulim uu^{-1}\bp(\epsilon) =u\bp(\epsilon)
    \end{align*}    

    for each $\epsilon\subseteq [d]$. Therefore, we get that $u^{[d]}\bp\in F^{[d]}$. Since $\bg(\epsilon)\bp(\epsilon)^{-1} \in F$, we have $\bg\in \cltau{\HK{d}} \cap F^{[d]}$. Thus,\begin{align*}
        \bx \in v(\cltau{\HK{d}} \cap F^{[d]})x_{0}^{[d]}\subseteq u^{[d]} \circ v(\cltau{\HK{d}} \cap F^{[d]})x_{0}^{[d]} = u^{[d]} \circ (\cltau{\HK{d}} \cap F^{[d]})x_{0}^{[d]}.
    \end{align*}

    Therefore, by \cref{prop: Ellis_group_cube}, we get that\begin{align*}
        (\pi^{[d]})^{-1}(y_{0}^{[d]}) = u^{[d]} \circ (\cltau{\HK{d}} \cap F^{[d]})x_{0}^{[d]} = u^{[d]} \circ \GG(\bQ^{[d]}(Y))x_{0}^{[d]}.
    \end{align*}

    Since $\pi$ is RIC, we have that $\pi^{[d]}$ is an open extension. Hence, by \cref{lemma: RIC_iff_open_preim_is_ucircFx0}, we conclude that $\pi^{[d]}: \bQ^{[d]}(X) \to \bQ^{[d]}(Y)$ is RIC.
\end{proof}

\begin{proposition}
    Let $\pi: (X, T) \to (Y, T)$ be a RIC weakly mixing extension of minimal systems. Then, $\pi^{[d]}: (\bQ^{[d]}(X),\mathcal{HK}^{[d]}(T)) \to (\bQ^{[d]}(Y),\mathcal{HK}^{[d]}(T))$ is a RIC weakly mixing extension.
\end{proposition}
\begin{proof}
    Let $y_{0}\in Y$ and $(x_{1},\dots,x_{2^{d+1}})\in R_{\pi}^{2^{d+1}}$, with $x_{1},\dots,x_{2^{d+1}}\in \pi^{-1}(y_{0})$, be a transitive point. By \cref{lemma: preimage_d_of_y_in_cube}, we have that $(x_{1},\dots,x_{2^{d}}),(x_{2^{d}+1},\dots,x_{2^{d+1}}) \in \bQ^{[d]}(X)$.
    Let $(\bx,\bx')\in R_{\pi^{[d]}}$. Therefore, we have $\pi^{[d]}(\bx)=\pi^{[d]}(\bx')=\by$ for some $\by\in \bQ^{[d]}(Y)$. Then, there exist a net $(\bt_{\lambda})_{\lambda\in\Lambda} \subseteq \mathcal{HK}^{[d]}(T)$ such that $\bt_{\lambda}y_{0}^{[d]} \to \by$. Since $\pi$ is RIC, $\pi^{[d]}$ is open and $(\pi^{[d]})^{-1}$ is continuous. Thus, we have that\begin{align*}
        \bt_{\lambda}(\pi^{[d]})^{-1}(y_{0}^{[d]}) \to (\pi^{[d]})^{-1}(\by).
    \end{align*}

    Therefore, there exist $(\bold{a}_{\lambda})_{\lambda\in\Lambda},(\bold{b}_{\lambda})_{\lambda\in\Lambda}$ nets in $(\pi^{[d]})^{-1}(y_{0}^{[d]})$ such that $\bt_{\lambda}\bold{a}_{\lambda} \to \bx$ and $\bt_{\lambda}\bold{b}_{\lambda} \to \bx'$. It follows from $(\bold{a}_{\lambda,1},\dots, \bold{a}_{\lambda,2^{d}},\bold{b}_{\lambda,1},\dots,\bold{b}_{\lambda,2^{d}})\in R_{\pi}^{2^{d+1}}$, where $\bold{a}_{\lambda}=(\bold{a}_{\lambda,1},\dots, \bold{a}_{\lambda,2^{d}})$ and $\bold{b}_{\lambda}=(\bold{b}_{\lambda,1},\dots,\bold{b}_{\lambda,2^{d}})$, that $\bold{a}_{\lambda}\in \overline{T(x_{1},\dots,x_{2^{d}})}$ and $\bold{b}_{\lambda}\in \overline{T(x_{2^{d}+1},\dots, x_{2^{d+1}})}$. Therefore, we have\begin{align*}
        (\bx,\bx')\in \overline{\mathcal{HK}^{[d]}(T)((x_{1},\dots,x_{2^{d}})(x_{2^{d}+1},\dots,x_{2^{d+1}}))}.
    \end{align*}
    
    Hence, we get that $R_{\pi^{[d]}}$ is transitive, completing the proof.
\end{proof}

By definition, the same results are obtained for distal and equicontinuous extensions. 

\subsection{Maximal distal factor of dynamical cubes}

In this subsection, we will show that the maximal distal factor and maximal factor of order $\infty$ of $\bQ^{[d]}(X)$ is $\bQ^{[d]}(X_{dis})$, where $X_{dis}$ is the maximal distal factor of $X$, and $\bQ^{[d]}(X/\RP^{[\infty]}(X))$, respectively, for any minimal system $(X, T)$. This provides an alternative proof to the one given by Qiu and Zhao in \cite{Qiu_Zhao_Maximal_nil_d_cubes:2021} for the maximal factor of order $\infty$ of dynamical cubes.

We will first prove the theorem for the maximal distal factor of dynamical cubes. The proof of this is analogous to the proof of \cite[Theorem 1.2]{Wu_Xu_Ye_Structure_saturated:2023} for the system $N_{d}(X)$ (see \cite{Wu_Xu_Ye_Structure_saturated:2023} for the definition of $N_{d}(X)$), using the saturatedness of dynamical cubes proved in \cref{prop: cube_saturated_open_ext_nil_infty}. For completeness, we include the proof. The proof makes use of the capturing operation, a kind of reverse orbit closure, which was introduced in \cite{Auslander_Glasner_distal_order:2002} to characterize the distal and equicontinuous structure relations in minimal systems.

Let $(X, T)$ be a minimal topological dynamical system and $K\subseteq X$. The capturing set of $K$ is $C(K)=\{x\in X: \overline{Tx}\cap K \neq \emptyset\}$. Let $R$ be a symmetric and reflexive relation, and let $\mathcal{E}(R)$ be the equivalence relation generated by $R$. Thus, $\mathcal{E}(R)=\bigcup\{R^{n}:n=1,2,\dots\}$, where\begin{align*}
    R^{n}=\{(x,z):\exists y_{1},\dots,y_{n-1} \text{ s.t } (x,y_{1}),(y_{1},y_{2}),\dots,(y_{n-1},z)\in R \}.
\end{align*}

Every system $(X,T)$ has a maximal distal factor. Moreover, there exists a closed, invariant equivalence relation $S_{dis}(X)$ such that the quotient $X/S_{dis}(X)$ is the maximal distal factor of $X$, $S_{dis}(X)$ is called the {\em distal structure relation} of $X$. The following results show some properties of the distal structure relation.

\begin{theorem}[{\cite[Theorem 4.5]{Auslander_Glasner_distal_order:2002}}]\label{thm: char_Sdis_capturing_set}\label{thm: S_dis_capturing_set}
    Let $(X, T)$ be a minimal topological dynamical system. Then, $S_{dis}(X)=C(\overline{\mathcal{E}(\overline{\P(X)}}))$.
\end{theorem}

\begin{theorem}[{\cite[Theorem 2.1]{Auslander_Glasner_distal_order:2002}}]\label{thm: image_of_Sdis}
    Let $\pi: (X, T) \to (Y, T)$ be an extension of minimal systems. Then, $\pi\times \pi (S_{dis}(X))=S_{dis}(Y)$.
\end{theorem}

Before we can prove the theorem, we need some preliminary results.

\begin{lemma}[{\cite[Lemma 4.6]{Wu_Xu_Ye_Structure_saturated:2023}}]\label{lemma: X_dis_is_Y_dis_proximal extension}
    Let $\pi: (X, T)\to (Y, T)$ be an extension of minimal systems. If $R_{\pi}\subseteq S_{dis}(X)$, then $X_{dis}=Y_{dis}$. In particular, this holds for proximal extensions.
\end{lemma}

\begin{lemma}[{\cite[Lemma 4.7]{Wu_Xu_Ye_Structure_saturated:2023}}]
    Let $\pi: (X, T)\to (Y, T)$ be an extension of minimal systems and $Z$ be a distal factor of $Y$. If the maximal distal factor of $X$ is $Z$, then $Z$ is also the maximal distal factor of $Y$.
\end{lemma}

\begin{lemma}\label{lemma: P_cube_if_proximal_closed}
    Let $(X, T)$ be a minimal topological dynamical system with proximal relation closed, and let $d\geq 1$ be an integer. Then,\begin{align*}
        \P((\bQ^{[d]}(X))=\{(\bx,\by)\in \bQ^{[d]}(X)\times \bQ^{[d]}(X): (\bx_{\epsilon},\by_{\epsilon})\in \P(X), \epsilon\subseteq [d]\}.
    \end{align*}
\end{lemma}

\begin{proof}
    Clearly, we have the following inclusion\begin{align*}
        \P((\bQ^{[d]}(X))\subseteq\{(\bx,\by)\in \bQ^{[d]}(X)\times \bQ^{[d]}(X): (\bx_{\epsilon},\by_{\epsilon})\in \P(X), \epsilon\subseteq [d]\}.
    \end{align*}

    Let $(\bx,\by) \in \bQ^{[d]}(X)\times \bQ^{[d]}(X)$ be such that $(\bx_{\epsilon},\by_{\epsilon}) \in \P(X)$. Consider $\pi:X \to X_{dis}$, where $X_{dis}$ is the maximal distal factor. Since $\pi$ is a proximal extension, by \cref{prop: pi_prox_pid_prox_in_cubes}, we have $\pi^{[d]}: \bQ^{[d]}(X)\to \bQ^{[d]}(X_{dis})$ is a proximal extension.  It follows from $\pi^{[d]}(\bx)=\pi^{[d]}(\by)$ that $(\bx,\by)\in \P(\bQ^{[d]}(X))$, completing the proof.
\end{proof}

\begin{theorem}\label{thm: distal_factor_cube}
    Let $(X, S)$ be a minimal topological dynamical system and let $d\geq 1$ be an integer. Then, $\bQ^{[d]}(X_{dis})$ is the maximal distal factor of $\bQ^{[d]}(X)$, where $X_{dis}$ is the maximal distal factor of $X$.
\end{theorem}

\begin{proof}

    Let $\pi: X\to X_{dis}$ be the factor map. By \cref{thm: AG_diagram}, there is a commutative diagram of extensions of minimal systems
     \[\begin{tikzcd}
	X && {X'} \\
	\\
	{X_{dis}} && {X_{dis}'}
	\arrow["\pi"', from=1-1, to=3-1]
	\arrow["{\theta'}"', from=1-3, to=1-1]
	\arrow["{\pi'}", from=1-3, to=3-3]
	\arrow["\theta", from=3-3, to=3-1]
\end{tikzcd}\]

where $\pi'$ is an open extension and $\theta$ and $\theta'$ are almost one-to-one extensions. Thus, we have the following diagram.

 \[\begin{tikzcd}
	\bQ^{[d]}(X) && {\bQ^{[d]}(X')} \\
	\\
	{\bQ^{[d]}(X_{dis})} && {\bQ^{[d]}(X_{dis}')}
	\arrow["\pi^{[d]}"', from=1-1, to=3-1]
	\arrow["{\theta'}^{[d]}"', from=1-3, to=1-1]
	\arrow["{\pi'}^{[d]}", from=1-3, to=3-3]
	\arrow["\theta^{[d]}", from=3-3, to=3-1]
        \arrow["\phi^{[d]}", from=1-3, to=3-1]
\end{tikzcd}\]

Let $(\bx, \by) \in R_{\phi^{[d]}}$, and let $\bx'= {\pi'}^{[d]}(\bx)$ and $\by'={\pi'}^{[d]}(\by)$. Observe that $(\bx',\by')\in R_{\theta^{[d]}}$. Since $\theta$ is almost one-to-one and $X_{dis}$ is distal, we have that $\P(X_{dis}')$ is closed. Therefore, by \cref{lemma: P_cube_if_proximal_closed}, we get that $(\bx ',\by')\in \P(\bQ^{[d]}(X_{dis}'))$. Then, there exists $\bold{v}\in J(M(\mathcal{HK}^{[d]}(S)))$ such that $\bold{v}\by'=\bold{v}\bx'$. Let $\bold{a}=\bold{v}\bx$ and $\bold{b}=\bold{v}\by$. Note that $(\bold{a}, \bold{b})\in R_{{\pi'}^{[d]}}$.

It follows from that $\theta'$ is a proximal extension and \cref{lemma: X_dis_is_Y_dis_proximal extension} that $X_{dis}$ is the maximal distal factor of $X'$. Since $X/\RP^{[\infty]}(X)$ is a factor of $X_{dis}$ and $\pi'$ is an open extension, by \cref{prop: cube_saturated_open_ext_nil_infty}, we get that $\bQ^{[d]}(X')$ is ${\pi'}^{[d]}$-saturated. Writing $\bold{a}=(a_{1},\dots,a_{2^{d}})$ and $\bold{b}=(b_{1},\dots,b_{2^{d}})$, it follows from the saturated property of $\bQ^{[d]}(X')$ that \begin{align*}
(a_{1},\dots,a_{i},b_{i+1},\dots,b_{2^{d}}) \in \bQ^{[d]}(X')
\end{align*}
for each $0\leq i \leq 2^{d}$. Note that $(a_{1},b_{1})\in S_{dis}(X')$, since $(\bold{a}, \bold{b})\in R_{\phi^{[d]}}$. Thus, by \cref{thm: S_dis_capturing_set}, we have \begin{align*}
    \overline{S(a_{1},b_1)} \cap \overline{\mathcal{E}(\overline{\P(X')})} \neq \emptyset.
\end{align*}

Therefore,
\begin{align*}
    \overline{\mathcal{HK}^{[d]}(S)(\bold{a}, (b_{1},a_{2},\dots,a_{2^{d}}))} \cap \overline{\mathcal{E}(\overline{\P(\bQ^{[d]}(X')})} \neq \emptyset.
\end{align*}

Thus, $(\bold{a}, (b_{1},a_{2},\dots,a_{2^{d}})) \in S_{dis}(\bQ^{[d]}(X'))$. Consider $((b_{1},a_{2},a_{3},\dots, a_{2^{d}}),(b_{1},b_{2},a_{3},\dots,a_{2^{d}}) \in \bQ^{[d]}(X')$. Similarly, we get that
\begin{align*}
    ((b_{1},a_{2},a_{3},\dots, a_{2^{d}}),(b_{1},b_{2},a_{3},\dots,a_{2^{d}})) \in S_{dis}(\bQ^{[d]}(X')).
\end{align*}

Therefore, inductively, we have \begin{align*}
    ((b_{1},\dots,b_{i-1},a_{i},\dots,a_{2^{d}}),(b_{1},\dots,b_{i-1},b_{i},a_{i+1},\dots,a_{2^{d}})) \in S_{dis}(\bQ^{[d]}(X'))
\end{align*}
for each $2\leq i \leq 2^{d}$. Since $S_{dis}(\bQ^{[d]}(X'))$ is an equivalence relation, we have $(\bold{a},\bold{b})\in S_{dis}(\bQ^{[d]}(X'))$. Note that for all $\bold{w_1},\bold{w_2}\in J(M(\mathcal{HK}^{[d]}(S)))$, $(\bold{w_1},\bold{w_2})$ leaves $S_{dis}(\bQ^{[d]}(X'))$ invariant. Considering $\bold{w_1},\bold{w_2}\in J(M(\mathcal{HK}^{[d]}(S)))$ such that $\bold{w_1} \bold{x}=\bold{x}$ and  $\bold{w_2}\bold{y}=\bold{y}$, we have that $\bold{x}=\bold{w_1}\bold{a}$ and $\bold{y}= \bold{w_2}\bold{b}$. We deduce that $(\bold{x},\bold{y})\in S_{dis}(\bQ^{[d]}(X'))$. It follows from $R_{\phi^{[d]}} \subseteq S_{dis}(\bQ^{[d]}(X'))$ that the maximal distal factor of $\bQ^{[d]}(X')$ is $\bQ^{[d]}(X_{dis})$. Hence, by \cref{lemma: X_dis_is_Y_dis_proximal extension} and  \cref{prop: pi_prox_pid_prox_in_cubes}, we can conclude that $\bQ^{[d]}(X_{dis})$ is the maximal distal factor of $\bQ^{[d]}(X)$.
\end{proof}

Using the last theorem, we have the following relation between $D(\mathcal{HK}^{[d]}(T))$ and $D$, where $D(\mathcal{HK}^{[d]}(T))$ is the Ellis group of the universal distal minimal system for the action of $\mathcal{HK}^{[d]}(T)$.

\begin{corollary}\label{cor: Ellis_group_distal_and_distal_cube_group}
Let $(X, S)$ be a minimal topological dynamical system and $d\geq 1$ be an integer. Then, $D(\mathcal{HK}^{[d]}(S))\bx = D^{[d]} \bx$ for each $\bx\in\bQ^{[d]}(X)$.
\end{corollary}

\begin{proof}
    Note that $\bQ^{[d]}(M_{dis})$ is a distal system, so $D(\mathcal{HK}^{[d]}(S)) \subseteq \GG(\bQ^{[d]}(M_{dis}))$. By \cref{lemma: Ellis_distal_group_cube_universal} and \cref{prop: Ellis_group_cube}, we have $\GG(\bQ^{[d]}(M_{dis})) = D^{[d]}$. Therefore, $D(\mathcal{HK}^{[d]}(S)) \subseteq D^{[d]}$.

    Let $A=\GG(X)$. It follows from \cref{thm: distal_factor_cube}, \cref{prop: Ellis_group_cube} and \cref{prop: Ellis_group_distal_factor} that\begin{align*}
        \GG(\bQ^{[d]}(X_{dis}))=\GG((\bQ^{[d]}(X))_{dis}) = (\cltau{\HK{d}} \cap A^{[d]})D(\mathcal{HK}^{[d]}(S)).
    \end{align*}

    By \cref{prop: Ellis_group_cube}, \cref{prop: Ellis_group_distal_factor} and \cref{lemma: Ellis_distal_group_cube_universal}, we have\begin{align*}
        \GG(\bQ^{[d]}(X_{dis}))=\cltau{\HK{d}} \cap A^{[d]}D^{[d]}.
    \end{align*}

    Let $\bg\in D^{[d]}$, $\bx\in \bQ^{[d]}(X)$ and $\bp\in M(\mathcal{HK}^{[d]}(S))$ be such that $\bx=\bp x_{0}^{[d]}$. Since $D$ is a normal subgroup, there is a $\bg'\in D^{[d]}$ such that $\bg\bp=u\bp\bg'$. It follows from $\GG(\bQ^{[d]}(X_{dis}))=\GG((\bQ^{[d]}(X))_{dis})$ that there are $\bh\in D(\mathcal{HK}^{[d]}(S))$ and $\bold{a}\in \cltau{\HK{d}}\cap A^{[d]}$ such that $\bg'=\bh\bold{a}$. Thus,\begin{align*}
        \bg\bx = \bg\bp x_{0}^{[d]} = u^{[d]}\bp\bg'x_{0}^{[d]} = u^{[d]}\bp\bh \bold{a}x_{0}^{[d]} = u^{[d]}\bp\bh x_{0}^{[d]}. 
    \end{align*}
    
    Since $D(\mathcal{HK}^{[d]}(S))$ is a normal subgroup, there is a $\bh'\in D(\mathcal{HK}^{[d]}(S))$ such that $\bh'\bp=u^{[d]}\bp\bh$. Therefore, we get that $\bg\bx = \bh'\bx$ and hence $D^{[d]}\bx\subseteq D(\mathcal{HK}^{[d]}(S))\bx$.    
\end{proof}

In \cite{Qiu_Zhao_topnilpotent_enveloping_nil:2022}, Qiu and Zhao showed that for a minimal distal system $(X, T)$ the maximal factor of order $\infty$ of $\bQ^{[d]}(X)$ is $\bQ^{[d]}(X/\RP^{[\infty]}(X))$. Soon after, they extended this result for any minimal system in \cite{Qiu_Zhao_Maximal_nil_d_cubes:2021}. In this paper, we provide an alternative proof of this result using \cref{thm: distal_factor_cube}.

\begin{theorem}[{\cite[Theorem 3.3]{Qiu_Zhao_topnilpotent_enveloping_nil:2022}}]\label{thm: factor_nil_infty_cube_distal_case}
    Let $(X, T)$ be a minimal distal topological dynamical system and $d\geq 1$ be an integer. Then, the maximal factor of order $\infty$ of $\bQ^{[d]}(X)$ is $\bQ^{[d]}(X/\RP^{[\infty]}(X))$.
\end{theorem}

Observe that we have the following lemma.
\begin{lemma}\label{lemma: Rpi_contain_RP_infty_nil_infty_coincide}
    Let $\pi: (X, T)\to (Y, T)$ be an extension of minimal systems. If $R_{\pi}\subseteq \RP^{[\infty]}(X)$, then $X/\RP^{[\infty]}(X)=Y/\RP^{[\infty]}(Y)$.
\end{lemma}

\begin{theorem}\label{thm: factor_nil_infty_cube}
    Let $(X, S)$ be a minimal topological dynamical system and $d\geq 1$ be an integer. Then, the maximal factor of order $\infty$ of $\bQ^{[d]}(X)$ is $\bQ^{[d]}(X/\RP^{[\infty]}(X))$.
\end{theorem}

\begin{proof}
    Since $S_{dis}(\bQ^{[d]}(X))$ is the smallest closed invariant equivalence relation containing $P(\bQ^{[d]}(X))$, we deduce that $S_{dis}(\bQ^{[d]}(X)) \subseteq \RP^{[\infty]}(\bQ^{[d]}(X))$. It follows from  \cref{lemma: Rpi_contain_RP_infty_nil_infty_coincide} that the maximal distal factor of $\bQ^{[d]}(X)$ has the same maximal factor of order $\infty$ as $\bQ^{[d]}(X)$. Similarly, we get $X/\RP^{[\infty]}(X) = X_{dis}/\RP^{[\infty]}(X_{dis})$. By \cref{thm: distal_factor_cube}, we have $(\bQ^{[d]}(X))_{dis} = \bQ^{[d]}(X_{dis})$. Thus, by \cref{thm: factor_nil_infty_cube_distal_case}, we conclude that the maximal factor of order $\infty$ of $\bQ^{[d]}(X_{dis})$ is $\bQ^{[d]}(X/\RP^{[\infty]}(X))$.
\end{proof}

 It is worth mentioning that in \cite{Qiu_Zhao_Maximal_nil_d_cubes:2021} the authors also showed that the maximal factor of order $j$ of  $\bQ^{[d]}(X)$ is $\bQ^{[d]}(X_j)$, where $X_j$ is the maximal factor of order $j$ of $X$. Currently, we are not able to give a proof for this fact using our methods. It would be interesting to provide a proof using the algebraic methods developed in this article.  Moreover, to provide an algebraic proof, it would suffice to show that $\GG(\bQ^{[d]}(X)/\RP^{[j]}(\bQ^{[d]}(X))) = \GG(\bQ^{[d]}(X_{j}))$ for any minimal system $(X,T)$ and any integers $d,j \geq 1$. Furthermore, it would also be of interest to provide an algebraic proof of \cref{thm: factor_nil_infty_cube_distal_case} and \cref{thm: factor_nil_infty_cube}. As before, to give an algebraic proof, it is sufficient to prove that $\GG(\bQ^{[d]}(X)/\RP^{[\infty]}(\bQ^{[d]}(X))) = \GG(\bQ^{[d]}(X/\RP^{[\infty]}(X)))$ for any minimal system $(X,T)$ and any integer $d \geq 1$. The difficulty in both problems lies in using the results of this paper to prove these algebraic equalities.

\end{document}